\documentclass[11pt,reqno,a4paper,english]{amsart}

\usepackage[utf8]{inputenc}  
\usepackage{dsfont}
\usepackage{xcolor}

\usepackage{enumerate}          
\usepackage{amssymb,amsmath,amsfonts,amsthm}            
\usepackage{pdfsync}
\usepackage{url}
\usepackage{hyperref}
\usepackage{anysize}
\usepackage{mathrsfs}
\usepackage{graphicx,float,color}
\usepackage{epsfig}
\usepackage{tikz}

\begin{document}
\theoremstyle{plain}
\newtheorem{theorem}{Theorem}[section]
\newtheorem{definition}[theorem]{Definition}
\newtheorem{proposition}[theorem]{Proposition}
\newtheorem{lemma}[theorem]{Lemma}
\newtheorem{corollary}[theorem]{Corollary}
\newtheorem{example}[theorem]{Example}
\newtheorem{conjecture}[theorem]{Conjecture}
\newtheorem{remark}[theorem]{Remark}
\newtheorem{claim}[theorem]{Claim}
\errorcontextlines=0

\renewcommand{\P}{\mathcal{P}}
\newcommand{\Id}{\text{Id}}
\newcommand{\R}{\mathbb{R}}
\newcommand{\M}{\mathbf{H}^m}
\newcommand{\Op}{{\rm Op}}
\newcommand{\J}{\mathcal{J}}
\newcommand{\Z}{\mathbb{Z}}
\newcommand{\N}{\mathbb{N}}
\newcommand{\supp}{\text{supp}}
\newcommand{\ord}{\text{ord}}
\newcommand{\e}{\varepsilon}
\newcommand{\bqn}{\begin{equation}}
\newcommand{\eqn}{\end{equation}}
\newcommand{\al}{\alpha}
\newcommand{\tor}{(2\mathbb{T})^{2}}
\newcommand{\er}{{\overrightarrow{e_{r}}}}
\newcommand{\et}{{\overrightarrow{e_{\theta}}}}
\newcommand{\red}{\color{red}}
\newcommand{\black}{\color{black}}
\renewcommand{\i}{\iota}
\newcommand{\todo}[1]{$\clubsuit$ {\tt #1}}
\newcommand{\B}{\mathcal{B}}
\newcommand{\BST}{\textbf{(BST) }}
\newcommand{\BSCT}{\textbf{(BSCT) }}
\newcommand{\Green}{\textbf{(Green) }}
\newcommand{\reg}{p}
\newcommand{\num}{k}
\renewcommand{\Vec}{{\rm Vec}}
\newcommand{\man}{W}
\newcommand{\I}{\mathcal{I}}

\setcounter{tocdepth}{1}

\title[Generic controllability of equivariant systems]{Generic controllability of equivariant systems and applications to particle systems and neural networks}
\author{Andrei Agrachev}
\address{Andrei Agrachev. SISSA, via Bonomea 265, 34136 Trieste, Italy}
 \email{agrachevaa@gmail.com}

\author{Cyril Letrouit}
\address{Cyril Letrouit. Laboratoire de Math\'ematiques d'Orsay, Universit\'e
  Paris-Sud, Universit\'e Paris-Saclay, B\^atiment~307, 91405
  Orsay Cedex \& CNRS UMR 8628}
 \email{cyril.letrouit@universite-paris-saclay.fr}

\date{\today}

\maketitle

\begin{abstract}
There exist many examples of systems which have some symmetries, and which one may monitor with symmetry-preserving controls. Since symmetries are preserved along the evolution, full controllability is not possible, and controllability has to be considered inside sets of states with same symmetries. We prove that generic systems with symmetries are controllable in this sense. 

This result has several applications, for instance:  (i) generic controllability of particle systems when the kernel of interaction between particles plays the role of a mean-field control;  (ii)  generic controllability for families of vector fields on manifolds with boundary; (iii) universal interpolation for neural networks architectures with ``generic" self-attention-type layers - a type of layers ubiquitous in recent neural networks architectures, e.g., in the Transformers architecture.

The tools we develop could help address various other questions of control of equivariant systems.
\end{abstract}

\tableofcontents

\section{Introduction and main results}
\subsection{A motivating example}\label{s:motivation}
Given $n$ points at positions $x_1,\ldots,x_n \in \R^d$ (a point cloud), their empirical measure is the probability measure $\mu=\frac1n \sum_{j=1}^n \delta_{x_j}$. We are interested in time-dependent evolutions in the set of point clouds, and more precisely in controlling these evolutions, which amounts to operating control in the set of empirical measures. To monitor these evolutions, we control the interactions of the points, assumed to be of the form
\begin{equation}\label{e:evolK}
\forall i\in\{1,\ldots,n\}, \quad x_i(0)=x_i^0\in\R^d \; \; \text{and} \; \; \frac{d}{dt}x_i(t)=K_t\Bigl(x_i(t),\frac1n \sum_{j=1}^n \delta_{x_j(t)}\Bigr)
\end{equation}
where $(K_t)_{t\in \R}$ is our control, and for any time $t\in\R$, $K_t$ belongs to a given  time-independent family $\mathcal{K}$. This is a particular type of \emph{mean-field} control. 

Said with words, the influence felt by $x_i(t)$ and generated by the points $x_j(t)$, $j\neq i$, is given by the time-dependent kernel $K_t$ which we control. Each kernel $K\in\mathcal{K}$ depends both on $x_i(t)$ and all other positions $x_j(t)$, but not on their labels: $K_t$ writes $K_t(x_i(t),\mu(t))$ where $\mu(t)=\frac1n\sum_{j=1}^n \delta_{x_j(t)}$. This form of monitoring of particle systems evolution through control of the kernel arises for instance in neural networks architectures, as we will see later. 

If $\tau$ denotes the permutation $\tau=(ij)$ acting on $(\R^d)^n$ by permuting the $i$-th and $j$-th particles, then this action commutes with the evolution \eqref{e:evolK}. This observation implies the important property that the evolution \eqref{e:evolK} induces an evolution in the set of empirical measures; only positions of the points matter, not their labels. Moreover, \eqref{e:evolK} and the induced evolution of empirical measures necessarily preserve the mass of individual points, since 
\begin{equation}\label{e:equal0equalt}
x_i(0)=x_j(0) \Rightarrow \forall t\in\R, \ x_i(t)=x_j(t)
\end{equation}
(here and in the sequel, we assume that the evolution \eqref{e:evolK} is globally well-posed).

Given a fixed time $T>0$ and initial and final empirical measures $\mu^0, \mu^1$, our goal is to choose  $K_t\in\mathcal{K}$ for each $t\in[0,T]$ in a way that the solution to \eqref{e:evolK} with $\mu(0)=\mu^0$ satisfies $\mu(T)=\mu^1$. To achieve this goal, the following mass-preservation constraint must hold, due to \eqref{e:equal0equalt}: to each point in the support of $\mu^0$ must correspond a point in the support of $\mu^1$ with the same mass, and this correspondance must be one-to-one. Equivalently,
\begin{equation}\label{e:preseq}
\begin{aligned}
\exists x_1,\ldots, x_n, y_1,\ldots,y_n \in \R^d,  \text{ such } & \text{that } \mu^0=\frac1n\sum_{j=1}^n \delta_{x_j}, \  \mu^1=\frac1n\sum_{j=1}^n \delta_{y_j}, \\
\text{and for any } i,j \in \{1,\ldots,n\},& \ x_i=x_j \text{ if and only if } y_i=y_j.
\end{aligned}
\end{equation}

Our aim in this paper is to prove that if the constraint \eqref{e:preseq} is satisfied, this control problem is generically feasible. More precisely, if $\mathcal{K}$ contains at least two elements, and that  they are ``generic" enough, the condition \eqref{e:preseq} is the only constraint that initial and final data must satisfy to achieve our goal of sending $\mu^0$ to $\mu^1$ through an evolution of the form \eqref{e:evolK}.

 It turns out that this problem may be rephrased abstractly as a control problem in a manifold $M$ endowed with the action of a compact Lie group $G$. In the above example, $M=(\R^d)^n$ and $G=\mathfrak{S}_n$ is the symmetric group which acts on $M$ by permuting copies of $\R^d$ (see Example \ref{e:particles}). As already observed, permutations commute with the evolution \eqref{e:evolK}. In the general case, we consider only evolutions in $M$ following vector fields which are equivariant under the action of $G$ (see definition below). Consequently, motion in $M$ is constrained to remain inside some strata. In the above example, each stratum gathers points clouds with the same repartition of mass, i.e., two point clouds belong to the same stratum if and only if their empirical measures $\mu^0$ and $\mu^1$ satisfy \eqref{e:preseq}. In the general case, strata are given by connected components of sets of points in $M$ whose isotropy groups (=stabilizers) are in the same conjugacy class.

Solving this abstract control problem requires to dive into the subtle geometry of equivariant dynamical systems. As a result, it offers a wide range of applications, sometimes far from our original example of particle systems. Also, we believe that the tools we develop to solve this problem could be useful to address other questions related to control of equivariant systems.

Section \ref{s:mainresults} is devoted to the formulation of our main results, and Section \ref{s:applications} sketches some of its applications, which are developed in more detail in Section \ref{s:applic}.

\subsection{Main results}\label{s:mainresults}
Let $M$ be a real analytic\footnote{in this paper, all real analytic manifolds are assumed to be paracompact and second countable.} manifold and let $G$ be a compact Lie group acting analytically on $M$. $M$ is thus called an analytic $G$-manifold, and 
$$
M_G:=M/G=\{Gq \mid q\in M\}
$$
is the orbit space, each set $Gq$ being an orbit, or a $G$-orbit. The set of $C^\infty$ vector fields on $M$ is denoted by $\Vec(M)$. In the sequel, for $s\in\N$, we use the notation $[s]=\{1,\ldots,s\}$.

The isotropy group at $q\in M$ is
$$
G_q=\{g\in G \mid gq=q\}.
$$
Notice that if $q,q'$ belong to the same $G$-orbit, then $G_q$ and $G_{q'}$ are conjugate. The action of $G$ on $M$ induces a natural stratification
\begin{equation}\label{e:stratifM}
M_G=\bigsqcup_{i\in\I_G} S_G^i
\end{equation}
where for each $i\in \I_G$, $S_G^i$ is a connected component of the set of all orbits whose representatives have isotropy groups conjugate to\footnote{this does not depend on the representative} some given subgroup $H_i$ of $G$.

While reading the rest of this section, it might be helpful to keep in mind the example of Section \ref{s:motivation}, summarized here.
\begin{example}\label{e:particles} In Section \ref{s:motivation}, $G=\mathfrak{S}_n$ acts on $M=(\R^d)^n$ through 
$$
\sigma\cdot (x_1,\ldots,x_n)=(x_{\sigma(1)},\ldots,x_{\sigma(n)}).
$$
Then $M_G$ may be identified with the set of $n$-points empirical measures, i.e., probability measures of the form $\mu=\frac1n\sum_{j=1}^n \delta_{x_j}$. 
Two elements of $M_G$ belong to the same stratum in the decomposition \eqref{e:stratifM} if the corresponding empirical measures, denoted for instance by $\mu^0$ and $\mu^1$, satisfy \eqref{e:preseq}. The number of strata is finite.
\end{example}
The stratification \eqref{e:stratifM} is sometimes called ``stratification by isotropy types". As recalled in Section \ref{s:densitytransverse}, $\I_G$ is countable and each stratum $S_G^i$ is a smooth submanifold. 

Given $g\in G$ we define the diffeomorphism $P_g$ on $M$ by $P_gq=gq$. The pushforward $Y=(P_g)_*X$ of a vector field $X$ through $P_g$ is given by $Y(q)=(dP_g)_{P_g^{-1}(q)}(X(P_g^{-1}(q)))$ for any $q\in M$. Equivariant vector fields are those which are compatible with the action of $G$: $X\in\Vec(M)$ is equivariant under the action of $G$ if for any $g \in G$ there holds
\begin{equation}\label{e:defequivariant}
(P_g)_* X=X.
\end{equation}
Any equivariant vector field induces a vector field on $M_G$. We denote by $\Vec^G(M_G)$ the set of vector fields induced on $M_G$ by $C^\infty$ equivariant vector fields on $M$. This set is equipped with the (induced) $C^\infty$ topology on compact sets, whose definition is recalled in Section \ref{s:stronger}. 

We are interested in controllability properties in the quotient set $M_G$, with vector fields in $\Vec^G(M_G)$. We will see in Section \ref{s:proofgenequiv} that this is tightly related, but not equivalent, with controllability in $M$ with equivariant vector fields (controllability in $M_G$ is slightly weaker). We prove in Lemma \ref{l:fundamental} that any element of $\Vec^G(M_G)$ is tangent to strata defined in \eqref{e:stratifM}. Therefore, any integral curve of $\Vec^G(M_G)$ is contained in a single stratum, which implies that it is not possible to connect points $q,q'\in M_G$ by integral curves of $\Vec^G(M_G)$ if $q,q'$ do not belong to the same stratum. However, it might be possible to connect them if the two points $q,q'$ belong to the same stratum. This observation motivates the following definition:

\begin{definition}[Controllability in strata]\label{d:continstrata}
Let $X_1,\ldots,X_k\in\Vec^G(M_G)$. We say that controllability holds  in strata if for any $q,q'\in M_G$ belonging to the same stratum in \eqref{e:stratifM}, there exist $m\in\N$, $t_1,\ldots,t_m\in \R$ and $i_1,\ldots,i_m\in[k]$ (not necessarily distinct) such that 
$$
q'=e^{t_1X_{i_1}}\circ \ldots \circ e^{t_mX_{i_m}}q
$$
where $e^{tX}$ denotes the flow at time $t$ of the vector field $X$ on $M_G$.
\end{definition}

Our first main result is the following:
\begin{theorem}\label{t:genequiv}
There exists for any $k\geq 2$ a set of $k$-uples $(X_1,\ldots,X_k)\in (\Vec^G(M_G))^k$ which is residual in $(\Vec^G(M_G))^k$ and for which controllability holds in strata.
\end{theorem}
This roughly means that controllability in strata holds for ``almost any" $k$-uples of vector fields in $\Vec^G(M_G)$. Since $\Vec^G(M_G)$ is infinite-dimensional, residual sets are an appropriate framework to state ``almost-sure" properties. 

We actually prove a slightly stronger statement than Theorem \ref{t:genequiv}; namely, we prove the controllability in the leaves of the foliation generated by the equivariant fields on $M\setminus M'$, where $M'$ is defined in Section \ref{s:proofgenequiv}. The precise statement is given in Theorem \ref{t:precisedstatement}. 

In view of applications\footnote{The main application we have in mind here is to self-attention layers of neural networks, see Section \ref{s:selfattention}. Neural networks with self-attention layers are designed to map billions of sequences to billions of target sequences. Self-attention layers are implemented for instance in the Transformers architecture \cite{vaswani}, whose success in machine learning calls for mathematical explanations.}, it is natural to extend our result to the simultaneous control of $N$ points in $M_G$. Simultaneous control (also called ``ensemble control") means that with a single control that is shared by all $N$ points in $M_G$, we seek to drive the $N$ initial points to their $N$ respective targets (see \cite{agrachevsarychev}). In Example \ref{e:particles}, this means driving $N$ empirical measures to $N$ other empirical measures, evolving each of them independently (i.e. the empirical measures do not interact with each other), but with the same interaction kernel $K_t$ which may depend on time.

\begin{definition}[Simultaneous controllability in strata]\label{d:continstrataensemble}
Let $X_1,\ldots,X_k\in\Vec^G(M_G)$ for some $k\geq 2$. We say that simultaneous controllability in strata (of dimension $\geq 2$) holds if for any  $N\in\mathbb{N}$, any $q_1,\ldots,q_N$, $ q_1',\ldots,q_N'\in M_G$ satisfying:
\begin{enumerate}[(i)]
\item for any distinct $i,j\in[N]$, $q_i\neq q_j$ and $q_i'\neq q_j'$
\item for any $j\in[N]$, $q_j$ and $q_j'$ belong to the same stratum in \eqref{e:stratifM}, and the dimension of this stratum is $\geq 2$
\end{enumerate}
the following conclusion holds: there exist $m\in\N$, $t_1,\ldots,t_m\in \R$ and $i_1,\ldots,i_m\in[k]$ (not necessarily distinct) such that 
$$
\forall j\in[N], \qquad q_j'=e^{t_1X_{i_1}}\circ \ldots \circ e^{t_mX_{i_m}}q_j.
$$
\end{definition}

We prove the following result, which  is a generalization to the equivariant framework of \cite[Theorem 3.2]{agrachevsarychev}:
\begin{theorem}\label{t:Nfold}
For any $k\geq 2$, there exists a set of $k$-uples of equivariant $C^\infty$-vector fields $(X_1,\ldots,X_k)$ which is residual in $(\Vec^G(M_G))^k$, and for which simultaneous controllability in strata holds.
\end{theorem}
\cite[Theorem 3.2]{agrachevsarychev} can be recovered by taking $G$ reduced to the identity. Our proof of Theorem \ref{t:Nfold} (and of Theorem \ref{t:genequiv}) is constructive, whereas the proof of \cite[Theorem 3.2]{agrachevsarychev} relied on the multijet transversality theorem as a black-box.

\subsection{Applications}\label{s:applications}
Theorems \ref{t:genequiv} and \ref{t:Nfold} have various applications, which are developed in Section \ref{s:applic}. 
\begin{enumerate}[(i)]
\item \emph{Control in manifolds with boundary} (Section \ref{s:boundary}). Indeed, any manifold with boundary may be written as the quotient of a manifold without boundary by a reflection. This application of Theorem \ref{t:genequiv} is technically the simplest where $G$ is non-trivial, since $G$ is just $\Z/2\Z$. 
\item  \emph{Control of the spectrum of symmetric (or Hermitian) matrices} (Section \ref{s:quantum}). We apply Theorem \ref{t:genequiv} to the case where $M$ is the set of symmetric matrices, and $G$ is the orthogonal group, acting by conjugation on $M$. Each element $A$ of $M/G$ may be identified with the spectrum of any of its representatives (=symmetric matrices), i.e., the collection of eigenvalues seen up to permutations. The stratum to which $A$ belongs depends on the cardinality of each packet of coincident eigenvalues of $A$.
\item \emph{Control of particle systems} (Section \ref{s:mn}). This covers the example presented in Section \ref{s:motivation}, and its generalizations. In this case the particles live in some manifold $W$, so $M=W^n$, and $G=\mathfrak{S}_n$ (the particles are indistinguishable). The orbit space $M_G$ is identified with the set of $n$-points empirical measures on $W$.
\item \emph{Universal interpolation for generic self-attention layers in neural networks} (Section \ref{s:selfattention}). This can be framed as a particular case of the previous application. In this case each element of $M_G$ represents for instance a sentence, each element of $W$ is a word embedding (``a token"), and the time-evolution corresponds to evolution across layers. The goal explained in Section \ref{s:motivation} of sending the initial empirical measure to the target one typically represents a translation task, realized sentence by sentence (and not word by word).
\item \emph{Control of quantum systems with symmetries}, notably symmetric Ising spin networks of $n$ spin $\frac12$ particles (Section \ref{s:quantumspin}). These networks of $n$ states evolve according to Hamiltonians which are invariant under permutations of the spins. This is a particular case of application (iii), in the case where $W$ is the unit sphere of $\mathbb{C}^2$, in which spins live.
\end{enumerate}

\subsection{Open questions}\label{s:open}
Here are a few open questions which we believe of particular interest:
\begin{enumerate}
\item generalize Theorem \ref{t:genequiv} to the case where $G$ is not compact. Natural examples are the groups of translations and homotheties (centered at the origin) in Euclidean spaces. 
\item generalize Theorem \ref{t:genequiv} (or the simpler theorem by Lobry \cite{lobry}, see Section \ref{s:biblio}) to the case where $M$ has infinite dimension. This would possibly have applications to control of measures, of diffeomorphisms, and control of the spectrum of general self-adjoint or Hermitian operators of infinite dimension. There already exist in the literature results on the genericity of the controllability of infinite-dimensional quantum systems, seen as a Schr\"odinger PDE, see for instance \cite{mason}.
\item in the present paper, we only study the controllability problem, but the \emph{optimal control} problem is also certainly worth studying, for instance for an ensemble of $N$ points on $M_G$. This is natural in view of application to neural networks architectures. See \cite{scagliotti} in the case where there is no group acting on $M$.
\end{enumerate}
Also, let us mention here that our assumption of analyticity on $M$ and $G$ (which is satisfied in all natural examples) is technical. Although we do not know how to avoid it, we do not believe that this assumption is fundamental.

\subsection{Bibliography} \label{s:biblio}
The idea of proving controllability for ``generic dynamics" as in Theorem \ref{t:genequiv} is not new: in \cite{lobry}, Lobry proved that for a generic family of $k\geq 2$ vector fields on a connected manifold $M$ without boundary, any couple of points of $M$ may be connected by an integral path of the family. An outcome of our approach is an extension of Lobry's result to manifolds with boundary, see Corollary \ref{c:withboundary}. Lobry's paper has been extended in \cite{agrachevsarychev} to the case of ensembles of points on $M$ evolving according to a shared open loop control; the points are not interacting with each other but they are driven by the same control. This has direct applications to universal interpolation for so-called ``neural ODEs". Theorem \ref{t:Nfold} in the present work generalizes \cite[Theorem 3.2]{agrachevsarychev} to the equivariant framework, which is relevant among other applications to neural networks equipped with self-attention layers (see Section \ref{s:selfattention}). Compared to \cite{lobry} and \cite{agrachevsarychev}, our proofs are constructive and do not use (multi-)jet transversality. Also, the fact that $M$ is endowed with a group action is not a mere additional technicality: our proofs require a detailed understanding of the structure of orbits and strata of $G$-manifolds, partly based on the so-called slice theorem. 

It is important to mention that there already exists a vast literature on equivariant dynamical systems, see e.g. the book \cite{fieldbook} for a detailed account. Controllability and observability of equivariant dynamical systems, which are basically systems with symmetries, are well-developed subjects, see for instance \cite{bonnabel1}, \cite{bonnabel2}, and \cite{mahony} for a review. To our knowledge, generic controllability has never been studied in this framework, and the applications which we propose are original. We believe our fine analysis of equivariant dynamics (e.g., Lemma \ref{l:fundamental}) is of independent interest and could be useful to address other questions related to control of equivariant systems.

More references on applications of our results are given in Section \ref{s:applic}.

\subsection{Organization of the paper.} The applications sketched in Section \ref{s:applications} are developed in Section \ref{s:applic}. 

We prove Theorem \ref{t:genequiv} in Section \ref{s:proofgenequiv}. For this, we do not use equivariant transversality theory\footnote{see \cite{fieldbook} for an exhaustive treatment of this theory.}. Using this theory might seem to be a natural lead since the papers \cite{lobry}, \cite{agrachevsarychev} rely on transversality theory, but it turns out to be very cumbersome since equivariant transversality theory has many subtleties. Instead, our proofs rely on averaging techniques and codimension computations, and, as already said, require a detailed understanding of the structure of orbits and strata of $G$-manifolds, partly based on the so-called slice theorem.

In Section \ref{s:ensembles}, we prove Theorem \ref{t:Nfold}, which is a generalization of the proof of Theorem \ref{t:genequiv}.

\subsection{Acknowledgments.} We would like to thank Claude Viterbo and Mike Field for early discussions on this project, and Domenico d'Alessandro, Benjamin Apffel, Ugo Boscain, Thomas Iadecola, Tony Jin, Antoine Levitt and Eugenio Pozzoli for discussions about quantum evolutions and quantum control. We also thank Alessandro Scagliotti for feedbacks on a previous version of this manuscript. This project has received funding from the European Research Council (ERC) under the European Union's Horizon 2020 research and innovation programme (grant agreement No. 945655). The second author would like to thank Luca Rizzi for his kind invitation to the SISSA in January 2024, where part of this work was carried out. 

\section{Applications}\label{s:applic}

This section develops the applications which have been sketched in Section \ref{s:applications}.
\subsection{Manifolds with boundary}\label{s:boundary}
Let $\widetilde{M}$ be a real-analytic manifold with smooth boundary $\partial \widetilde{M}\neq \emptyset$. We denote by $\mathcal{U}$ the set of $C^\infty$-vector fields on $\widetilde{M}$, defined up to the boundary $\partial \widetilde{M}$, and which are tangent to $\partial \widetilde{M}$. This set is endowed with the $C^\infty$ topology on compact sets. Corollary \ref{c:withboundary} asserts that for a generic $k$-uple ($k\geq 2$) of elements of $\mathcal{U}$, any two points belonging either to the same connected component of $\partial\widetilde{M}$ or to the same connected component of the interior of $\widetilde{M}$ may be connected by an integral path of the vector fields. This result is a generalization to manifolds with boundary of a result due to Lobry \cite{lobry} (which does not require the analyticity assumption).
\begin{corollary}[Generic control on manifolds with boundary]  \label{c:withboundary}
Let $\widetilde{M}$ be a real-analytic and connected manifold with smooth boundary $\partial \widetilde{M}$. Then for any integer $\num\geq 2$ there exists a residual set of $k$-uples $(X_1,\ldots,X_k)\in\mathcal{U}^\num$ for which the following property holds. For any $q,q'$ belonging either to the same connected component of $\partial \widetilde{M}$ or to the same connected component of the interior of $\widetilde{M}$, there exist $m\in\N$, $t_1,\ldots,t_m\in \R$ and $i_1,\ldots,i_m\in[k]$ (not necessarily distinct) such that $q'=e^{t_1X_{i_1}}\circ \ldots \circ e^{t_mX_{i_m}}q$.
\end{corollary}

Let us explain how Corollary \ref{c:withboundary} follows from Theorem \ref{t:genequiv}. To any compact manifold $\widetilde{M}$ with smooth boundary is naturally associated a $G$-manifold $M$, with $G=\mathbb{Z}_2$ (here $\mathbb{Z}_2=\mathbb{Z}/2\mathbb{Z}$), constructed as follows. First, the double $M$ of $\widetilde{M}$ is formed by gluing together two copies of $\widetilde{M}$ along their common boundary. There is a natural action of $\mathbb{Z}_2$ by reflection on the manifold $M$ fixing the common boundary, and sending each point of the first copy of $\widetilde{M}$ to the same point in the second copy of $\widetilde{M}$, and vice versa. Then, the quotient space $M/\mathbb{Z}_2$ can be identified with $\widetilde{M}$. 

The strata of $\widetilde{M}\simeq M/\mathbb{Z}_2$ are the connected components of the boundary and of the interior of $\widetilde{M}$. 
As an illustration, if $\widetilde{M}$ is a disk, then $M$ is a 2-dimensional sphere, and the set of fixed points through the mirror action is an equator of $M$. Equivariant vector fields on $M$ are tangent to this equator. Then Theorem \ref{t:genequiv} applied to $\widetilde{M}$ yields Corollary \ref{c:withboundary}. Of course, there exists an ``ensemble version" of Corollary \ref{c:withboundary}, which follows from  Theorem \ref{t:Nfold}.

\subsection{Spectrum of matrices}\label{s:quantum} 
Our results also have applications to control of symmetric and Hermitian matrices.
Let us consider the natural action of the orthogonal group $G=\mathcal{O}_n$ on the space of symmetric matrices $\mathcal{S}_n$: any $P\in \mathcal{O}_n$ acts by $\mathcal{S}_n\ni S\mapsto PSP^{\top}$. By diagonalization of symmetric matrices, we identify an element of $\mathcal{S}_n/\mathcal{O}_n$ with the empirical measure of the eigenvalues of any of its representatives in $\mathcal{S}_n$, i.e., 
$$
\mathcal{S}_n/\mathcal{O}_n\simeq \R^n/\mathfrak{S}_n\simeq \mathcal{M}_n(\R)
$$
where $\mathfrak{S}_n$ denotes the symmetric group over $n$ elements, acting on $\R^n$ by permuting the coordinates, and $\mathcal{M}_n(\R)$ is the set of empirical measures over $n$ points in $\R$, that is, the set of probability measures on $\R$ of the form $\mu=\frac1n \sum_{j=1}^n \delta_{x_j}$.

Given $S\in \mathcal{S}_n/\mathcal{O}_n$, we denote by $\lambda_1 < \ldots < \lambda_{n(S)}$ the \emph{distinct} eigenvalues of any representative, and by $m_1, \ldots, m_{n(S)}$ their respective multiplicities. The tuple $(m_1, \ldots, m_{n(S)})$ is called the ordered multiplicities of $S$. Then $S,S'\in\mathcal{S}_n/\mathcal{O}_n$ belong to the same stratum if and only if their ordered multiplicities coincide, i.e., $n(S)=n(S')$ and $m_i(S)=m_i(S')$ for any $1\leq i\leq n(S)=n(S')$. As a side remark, recall that it is known since Von Neumann and Wigner \cite{vonneumann} that in $\mathcal{S}_n$, the set of matrices with two coincident eigenvalues has codimension $2$.

We denote by $\mathcal{W}$ the set of vector fields on $\mathcal{S}_n/\mathcal{O}_n\simeq \mathcal{M}_n(\R)$ induced by $\mathcal{O}_n$-equivariant $C^\infty$ vector fields on $\mathcal{S}_n$. Along each integral curve in $\mathcal{S}_n/\mathcal{O}_n$ of the family $\mathcal{W}$, ordered multiplicities are preserved. Our result may be phrased as follows: for generic $k$-uples ($k\geq 2$) of elements of $\mathcal{W}$, it is possible to transfer by appropriate composition of the flows of these vector fields any empirical measure of eigenvalues to any other empirical measure of eigenvalues with the same ordered multiplicities.
\begin{corollary}[Generic control of the spectrum of symmetric matrices]\label{c:quantum}
For any integer $\num\geq 2$ there exists a residual set of $k$-uples $(X_1,\ldots,X_\num)\in\mathcal{W}^\num$, for which the following property holds. For any $S,S'\in\mathcal{S}_n/\mathcal{O}_n\simeq \mathcal{M}_n(\R)$ whose ordered multiplicities coincide, there exist $m\in\N$, $t_1,\ldots,t_m\in \R$ and $i_1,\ldots,i_m\in[k]$ (not necessarily distinct) such that $S'=e^{t_1X_{i_1}}\circ\ldots\circ e^{t_mX_{i_m}}S$.
\end{corollary}
This statement is a consequence of Theorem \ref{t:genequiv}, and there exists an ``ensemble version" which follows from Theorem \ref{t:Nfold}. Also, an analogous statement holds for the natural action of the unitary group on the space of $n\times n$ Hermitian matrices.

\subsection{Particle systems} \label{s:mn}
We develop now the application of our results to (mean-field) control of particle systems, making rigorous Section \ref{s:motivation}. If $n\in\N$ and $\man$ is a manifold, the set $\mathcal{M}_n\left(W\right)$ of empirical measures over $n$ points in $W$, that is, the set of measures of the form $\mu=\frac1n \sum_{j=1}^n \delta_{x_j}$, also naturally carries a $G$-manifold structure. It is isomorphic to the quotient of $W^n$ by the action of the symmetric group $\mathfrak{S}_n$ given by
\begin{equation}\label{e:actionpermutation}
\sigma:(x_1,\ldots,x_n)\mapsto (x_{\sigma(1)},\ldots,x_{\sigma(n)})
\end{equation}
for $\sigma\in\mathfrak{S}_n$.
The isotropy group $G_q$ as $q=(x_1,\ldots,x_n)$ is not reduced to the identity if and only if at least two of the $x_i$'s are equal.

In the sequel we assume that $W$ is real-analytic, connected and of dimension $\geq 2$ (the case where $\dim(W)=1$ is actually treated in Section \ref{s:quantum}). Two points $q,q' \in W^n$, $q=(x_1,\ldots,x_n)$, $q'=(x_1',\ldots,x_n')$ have conjugate isotropy groups if and only if there exists $h:W\rightarrow W$ and $\sigma\in\mathfrak{S}_n$ such that $h(x_\ell)=x'_{\sigma(\ell)}$ for any $\ell\in[n]$, in other words if and only if the numbers of pairs/triples/quadruples/... of points among $x_1,\ldots,x_n$ which are equal coincide with the same numbers computed for $x_1',\ldots,x_n'$. This condition is actually necessary and sufficient for $q$ and $q'$ to belong to the same stratum: since $W$ is connected and of dimension $\geq 2$, it is easy to construct a smooth path from $\frac1n\sum_{j=1}^n \delta_{x_j}$ to $\frac1n\sum_{j=1}^n \delta_{x_j'}$.

It is possible to give an analytic characterization of $\mathfrak{S}_n$-equivariant vector fields on $W^n$. For this we denote by $\mathcal{M}^{\bullet}_n(W)$  the set of couples $(x,\mu)$ where $\mu\in\mathcal{M}_n(W)$ and $x\in {\rm supp}(\mu)$. Then equivariant vector fields are in one-to-one correspondance with functions $f:\mathcal{M}^{\bullet}_n(W)\rightarrow T_\bullet W$ where this notation means that $f(x,\mu)\in T_xW$ for $x\in {\rm supp}(\mu)$. Let us describe this one-to-one correspondance. If $f:\mathcal{M}^{\bullet}_n(W)\rightarrow T_\bullet W$, then for $x=(x_1,\ldots,x_n)\in \man^n$, we define
$$
V(x)=(f(x_1,\mu),\ldots,f(x_n,\mu))\in T_x\man^n, \qquad {\rm where} \ \mu =\frac1n \sum_{j=1}^n \delta_{x_j}.
$$
It is immediate to verify that $V$ is an equivariant vector field on $W^n$.  Conversely, if $V=(V_1,\ldots,V_n)$ is an equivariant vector field, for $(x,\mu)\in\mathcal{M}^{\bullet}_n(\man)$ we set
\begin{equation}\label{e:Vy}
f(y,\mu)=V_1(y,x_2,\ldots,x_n)\in T_y\man
\end{equation}
where we have written $\mu=\frac1n (\delta_y+\sum_{j=2}^n\delta_{x_j})$. The right-hand side in \eqref{e:Vy} does not depend on the order in which we put $x_2,\ldots,x_n$ since $V$ is equivariant. Therefore \eqref{e:Vy} yields a well-defined $f:\mathcal{M}_n^\bullet \rightarrow T_\bullet W$.

We denote by $\mathcal{V}$ the set of functions $f:\mathcal{M}^{\bullet}_n(W)\rightarrow T_\bullet W$ such that the equivariant vector field on $W^n$ associated with $f$ is $C^\infty$ and generates a globally defined flow. The set $\mathcal{V}$ is endowed with the $C^\infty$ topology on compact sets.

According to the above characterization of equivariant vector fields on $W^n$, Theorem \ref{t:genequiv} reads as follows in this context:
\begin{corollary}[Generic control of particle systems]\label{c:particles}
For any integer $\num\geq 2$ there exists a residual set of tuples $(f_1,\ldots,f_k)\in\mathcal{V}^\num$ for which the following property holds. For any $\mu^0,\mu^1\in \mathcal{M}_n(W)$ in the same stratum, written as
\begin{equation}\label{e:mu0mu1}
\mu^0=\frac1n \sum_{r=1}^n \delta_{x_r^0} \qquad \text{and} \qquad  \mu^1=\frac1n \sum_{r=1}^n \delta_{x_r^1} 
\end{equation}
there exist $m\in\N$, $0=t_0<t_1< \ldots<t_m\in \R$, indices $i_1,\ldots,i_m\in[k]$ (not necessarily distinct), and signs $\varepsilon_1,\ldots,\varepsilon_m\in\{-1,1\}$ such that the unique solution to  the system of coupled ODEs
\begin{equation}\label{e:coupledODEs}
\forall \ell\in[n],\, \forall j\in[m], \, \forall t\in [t_{j-1},t_j[, \qquad \frac{d}{dt} x_\ell(t)=\varepsilon_jf_{i_j}(x_\ell(t),\mu(t)), \qquad \mu(t)=\frac1n \sum_{r=1}^n \delta_{x_r(t)}
\end{equation}
with initial values $(x_1(0),\ldots,x_n(0))=(x_1^0,\ldots,x_n^0)$ satisfies 
\begin{equation}\label{e:terminal}
\frac1n \sum_{r=1}^n \delta_{x_r(t_m)}=\mu^1.
\end{equation}
\end{corollary}
With the stronger Theorem \ref{t:precisedstatement}, it is even possible to impose that $x_r(t_m)=x_r^1$ for any $r\in[n]$ instead of the weaker condition \eqref{e:terminal} (but for this it is necessary that the points $x_r^0,x_r^1$ in the writing \eqref{e:mu0mu1} are numbered in a way that $x_r^0=x_{r'}^0$ if and only if $x_r^1=x_{r'}^1$, which is possible since $\mu^0$ and $\mu^1$ belong to the same stratum).

There are other natural group actions on $W^n$ for which Theorem \ref{t:genequiv} has natural corollaries, for instance the action by rotations and/or reflections when $W=\R^d$. We leave the precise statements to the reader. We also mention that the same $G$-manifold structure has been used to model the geometry of  chords of music instruments, see e.g. \cite{music}.

\subsection{Universal interpolation with generic self-attention layers}\label{s:selfattention}

The application of Theorem \ref{t:genequiv} (or Theorem \ref{t:Nfold}) to particle systems, developed in the previous section, is relevant in machine learning. One of our motivations for this paper is actually to understand the possibilities of approximation offered by  a relatively new neural network architecture introduced in \cite{vaswani}, called Transformers, which play nowadays a central role in the inner workings of large language models (the last letter in ``Chat-GPT" stands for Transformers): more precisely, we would like to study which classes of functions these neural networks architectures are able to approximate. If this class is large, it suggests that the architecture is able to handle many different types of data and problems. 

Approximation and interpolation properties of some neural networks (see \cite{cheng} for precise definitions) are known to be equivalent to controllability properties of some non-linear systems of ODEs in discrete or continuous time. In the past 5 years, tools from geometric control like Lie bracketing have therefore been used to study the controllability properties of so-called ResNets (standing for ``residual neural networks") and their continuous-time version called neural ODEs, see e.g. \cite{cuchiero}, \cite{tabuada}, \cite{agrachevsarychev}.

However, Transformers are not of the same nature as ResNets and neural ODEs. The main difference is that they incorporate self-attention layers, which may be seen from the mathematical point of view as interacting particle systems or evolutions in the set of (empirical) measures (see \cite{vuckovic}, \cite{sinkformer}, \cite{transformers1}, \cite{transformers2}). The results of the present paper give insights on the approximation/interpolation properties of (generalized) self-attention layers, if one forgets about the rest of the architecture of Tranformers which, in addition to self-attention layers, usually incorporate also normalization layers and multi-layer perceptrons.

Our definition of self-attention layers is much broader than the specific self-attention layers used in practice: we call ``generalized self-attention layer" any vector field on $(\R^d)^n$ which may be written for some $f:\mathcal{M}^{\bullet}_n(\R^d)\rightarrow T_\bullet \R^d$ as
\begin{equation}\label{e:genSAlayer}
(x_1,\ldots,x_n)\mapsto (f(x_1,\mu),\ldots,f(x_n,\mu)), \qquad \text{where }\ \mu=\frac1n \sum_{r=1}^n \delta_{x_r}.
\end{equation}
In other words, generalized self-attention layers are nothing else than the infinitesimal-time version of a  permutation-equivariant sequence-to-sequence maps. In discrete time, they would take the form of a discrete system of coupled ODEs
\begin{equation}\label{e:discrete}
x_i((m+1)\Delta t)=x_i(m\Delta t)+(\Delta t)f(x_i(m\Delta t),\mu(m\Delta t)),  \qquad  \mu(m\Delta t)=\frac1n \sum_{r=1}^n \delta_{x_r(m\Delta t)}
\end{equation}
for some fixed $\Delta t>0$, and any $m\in\N$. In the terminology of neural networks, the term $x_i(m\Delta t)$ in the right-hand side is called a skip-connection.

Our result, which is a genericity result, works for ``almost all generalized self-attention layers", but does not say anything about the interpolation properties of the specific self-attention layers used in practice, which correspond to the choice of functions $f$ of the form
\begin{equation}\label{e:inpractice}
f\Bigl(x,\frac1n \sum_{j=1}^n \delta_{x_j}\Bigr)=\sum_{h=1}^H\frac{\sum_{j=1}^n e^{\langle Q_h x,K_h x_j\rangle}V_h x_j}{\sum_{j=1}^n e^{\langle Q_hx,K_hx_j\rangle}}
\end{equation}
for some $d\times d$ matrices $Q_h,K_h,V_h$ (see below for literature on this problem).

The study of the interpolation properties of generalized self-attention layers boils down to a problem of controllability of interacting particle system of the form presented in Corollary \ref{c:particles} (or Corollary \ref{c:selfattention} below). Our goal is not to control an interacting particle system with fixed interaction kernel by acting on a subset of particles, which is a classical question; instead, in our problem, the controls are directly given by a family of interaction kernels.

One particular feature of generalized self-attention layers (or equivalently interacting particle systems) is that they are equivariant with respect to permutation of particles. Therefore, approximation/interpolation properties are considered in the class of permutation-equivariant maps. Recall that equivariant neural networks are of particular interest because they maintain their performance even when the input data undergoes certain transformations, such as rotations, translations, or scaling, and they have a reduced number of parameters compared to non-equivariant counterparts, as they exploit the inherent symmetries in the data. 

In the context of particle systems, Theorem \ref{t:Nfold} takes the form of Corollary \ref{c:selfattention} below.  Interpreted as a universal interpolation result, it says that if we take $k\geq 2$ sufficiently generic maps (=layers) of the form \eqref{e:genSAlayer}, then it is possible for any given initial set of $N$ point clouds and any given final set of $N$ (target) point clouds, to compose these $k$ maps\footnote{At the level of the discretized equation \eqref{e:discrete}, this would mean alternating between different $f$ for different values of $m$.}, possibly many times and in some appropriate order, in a way to move each of the $N$ point clouds to its target. In other words,
\begin{center}
\emph{universal interpolation is a generic property of $k$-uples of generalized self-attention layers.}
\end{center}
\begin{corollary}\label{c:selfattention}
For any integer $\num\geq 2$ there exists a residual set of tuples $(f_1,\ldots,f_k)\in\mathcal{V}^\num$ for which the following property holds. Let $N\in\N$ and $\mu_1^0,\ldots,\mu_N^0,\mu_1^1,\ldots,\mu_N^1$ be distinct elements of $\mathcal{M}_n(W)$ such that for any $h\in[N]$, $\mu_h^0$ and $\mu_h^1$ are in the same stratum. For any $h\in[N]$ we write
$$
\mu_h^0=\frac1n\sum_{r=1}^n \delta_{x_{h,r}^0} \qquad \text{and} \qquad  \mu_h^1=\frac1n \sum_{r=1}^n \delta_{x_{h,r}^1}.
$$
Then there exist $m\in\N$, $0=t_0<t_1< \ldots<t_m\in \R$, indices $i_1,\ldots,i_m\in[k]$ (not necessarily distinct) and signs $\varepsilon_1,\ldots,\varepsilon_m\in\{-1,1\}$ such that for any $h\in[N]$, the unique solution to  the system of coupled ODEs
\begin{equation}\label{e:coupledODEs22}
\forall \ell\in[n],\, \forall j\in[m], \, \forall t\in [t_{j-1},t_j[, \quad \frac{d}{dt} x_{h,\ell}(t)=\varepsilon_j f_{i_j}(x_{h,\ell}(t),\mu_h(t)), \quad \mu_h(t)=\frac1n \sum_{r=1}^n \delta_{x_{h,r}(t)}
\end{equation}
with initial condition $(x_{h,1}(0),\ldots,x_{h,n}(0))=(x_{h,1}^0,\ldots,x_{h,n}^0)$ satisfies 
$$
\frac1n\sum_{r=1}^n\delta_{x_{h,r}(t_m)} = \mu^1_h.
$$
\end{corollary}  

Compared to previous results in the literature \cite{yun}, \cite{sumformer}, our result does not need to incorporate multi-layer perceptrons to achieve universal interpolation. Besides, it is not restricted to the Euclidean setting and works for data on general manifolds $W$, thus it may be considered as a statement in geometric deep learning (see \cite{bronstein}, and \cite{gerken} for geometric deep learning with equivariant neural networks). However, its main drawback is that it does not rely on self-attention layers used in practice, namely \eqref{e:inpractice}, but on ``generic" self-attention layers, of the general form \eqref{e:genSAlayer}. Also, it deals with universal interpolation instead of universal approximation.

It would be relevant to prove an analogue of Corollary \ref{c:particles} in restricted classes of particle systems. We have in mind the following type of statement: for generic $k$-uples of functions chosen in a class of evolutions $\mathcal{F}$ smaller than $\mathcal{V}_\reg$, the same conclusion as in Corollary \ref{c:particles} holds. For instance, does it hold when $\mathcal{F}$ is the family of evolutions driven by pairwise interactions? I.e., each $f_i$ in \eqref{e:coupledODEs} is of the form 
$$
f_i\Bigl(x,\frac1n \sum_{j=1}^n \delta_{x_j}\Bigr)=\sum_{j=1}^n K_i(x,x_j)
$$
for some kernel $K_i$.

Another family $\mathcal{F}$ with meaningful applications is given by the attention dynamics in Transformers without layer normalization, i.e., each $f_i$ would be of the form \eqref{e:inpractice} (see \cite{sinkformer}, \cite{transformers1}, \cite{transformers2}) for some $H\in\N$, and some $d\times d$ matrices $Q_h,K_h,V_h$.

Finally, let us mention that there are other equivariant neural networks architectures for which universal approximation theorems in the class of equivariant maps have already been proved, see for instance \cite{keriven} for Graph Neural Networks. Our results might also be applicable to this setting.

\subsection{Control of quantum systems with symmetries}\label{s:quantumspin}

Theorem \ref{t:genequiv} also applies to quantum systems controlled through Hamiltonians which display some symmetries. Let us develop one example. The papers \cite{GHZ}, \cite{alessandro} study symmetric Ising spin network where $n$ spin $\frac12$ particles (i.e., vectors in the unit sphere of $\mathbb{C}^2$) interact through a sum of an uncontrolled permutation-invariant Hamiltonian $H_0$, with some controlled permutation-invariant Hamiltonians $H_j$, $j=1,\ldots,k$:
\begin{equation}\label{e:perminha}
H(t)=H_0+\sum_{j=1}^k u_j(t)H_j
\end{equation}
where $u=(u_1,\ldots,u_k):[0,+\infty)\rightarrow\R^k$ is the control.
What we call here a permutation-invariant Hamiltonian is a Hamiltonian which is invariant under permutation of the spins, for instance
$$
H_{zz}=\sum_{1\leq k< m\leq n} \mathbf{1}\otimes \ldots \otimes \mathbf{1}\otimes \underbrace{\sigma_z}_{\text{$k^{\rm th}$ place}} \otimes  \mathbf{1}\otimes \ldots \otimes \mathbf{1}\otimes \underbrace{\sigma_z}_{\text{$m^{\rm th}$ place}} \otimes  \mathbf{1}\otimes \ldots \otimes \mathbf{1}
$$
or else
$$
H_{xyz}=\sum_{i,j,k \text{ distinct }} \mathbf{1}\otimes \ldots \otimes \mathbf{1}\otimes \underbrace{\sigma_x}_{\text{$i^{\rm th}$ place}} \otimes  \mathbf{1}\otimes \ldots \otimes \mathbf{1}\otimes \underbrace{\sigma_y}_{\text{$j^{\rm th}$ place}} \otimes  \mathbf{1}\otimes \ldots \otimes \mathbf{1}\otimes  \underbrace{\sigma_z}_{\text{$k^{\rm th}$ place}} \otimes  \mathbf{1}\otimes \ldots \otimes \mathbf{1}
$$
where
$$
\sigma_x:=\begin{pmatrix}0&1 \\ 1&0\end{pmatrix}, \qquad \sigma_y:=\begin{pmatrix}0&i \\ -i&0\end{pmatrix},\qquad \sigma_z:=\begin{pmatrix}1&0\\ 0&-1\end{pmatrix}
$$
are the Pauli matrices, and $\mathbf{1}$ is the $2\times 2$ identity matrix. We will say that two networks $q_1=(q_1^{(1)},\ldots,q_1^{(n)})$ and $q_2=(q_2^{(1)},\ldots,q_2^{(n)})$ of $n$ labelled spin $\frac12$ particles (each $q_i^{(j)}$ belongs to the unit sphere of $\mathbb{C}^2$) have same symmetries if $q_1^{(i)}=q_1^{(j)} \Leftrightarrow q_2^{(i)}=q_2^{(j)}$, for any $i,j\in[n]$. Our results (precisely, Theorem \ref{t:genliebracket22}) imply that if we are given at least two generic permutation-invariant Hamiltonians $H_1,\ldots,H_j$, $j\geq 2$, and two networks $q_1,q_2$ of $n$ spin $\frac12$ particles having same symmetries, then there exists a path from $q_1$ to $q_2$ corresponding to some choice of control $u(t)$ in \eqref{e:perminha}. In other words, subspace controllability (see \cite[Section 1]{alessandro23}) generically holds.

The focus of \cite{alessandro} is different but related: the authors give explicit examples of permutation-invariant Hamiltonians for which they are able to prove controllability and to design control laws. 

More generally, concerning the Lie bracket approach to control of quantum systems, we also mention the paper \cite{logicgate}, in which it is shown that generically, a quantum logic gate with two or more inputs is computationally universal, i.e., copies of the gate can be ``wired together" to effect any desired logic circuit, and to perform any desired unitary transformation on a set of quantum variables. This result exactly corresponds to Lobry's theorem \cite{lobry}, applied in the particular context of unitary transformations: in other words, it asserts that if one can apply some Hamiltonians (at least two) repeatedly to a few variables at a time one can in general effect any desired unitary time evolution on an arbitrarily large number of variables.

\section{Proof of Theorem \ref{t:genequiv}} \label{s:proofgenequiv}

This section is devoted to the proof of Theorem \ref{t:genequiv}. We actually prove a slightly stronger statement, given in Theorem \ref{t:precisedstatement}. In the whole paper, we work under the assumptions that $M$ is a real analytic manifold and $G$ is a compact Lie group acting analytically on $M$. 

\subsection{A stronger statement}\label{s:stronger}
This subsection is devoted to introducing notation and stating Theorem \ref{t:precisedstatement}.

Orbits, i.e., sets of the form $Gx$, are closed submanifold (see e.g. \cite[Corollary 3.1.17]{fieldbook}). We denote by $M'$ the set of points $x\in M$ at which all equivariant vector fields are tangent to the orbit $Gx$. This set $M'$ is closed and $G$-invariant. 
\begin{example}
If $G$ acts transitively on $M$, then $M'=M$. Even if $G$ is finite, the set $M'$ may be non-empty: for instance, if $M$ is the unit circle in $\R^2$ and $G$ acts by permuting coordinates, then $M'$ has cardinal $2$.
\end{example}
Although the above example shows that $M'$ is not empty in general, there holds $M'=\emptyset$ in all applications listed in Section \ref{s:applic} and in Example \ref{e:particles}. When $M'\neq \emptyset$, controllability is not necessarily possible inside connected components of $M'$, as shown in Example \ref{e:pascontM'}, therefore we need to adapt the definition of controllability, see Definition \ref{d:continleaves}.

We denote by $\Vec^G(M)$ the set of equivariant $C^\infty$ vector fields on $M$, equipped with the $C^\infty$ topology on compact sets defined as follows. For any chart $(\phi,U)$ of $M$, any compact set $K\subset U$, any $X\in \Vec^G(M)$, $k\in\N$ and $\varepsilon>0$, we consider the set $\mathcal{U}(X,\phi,K,\varepsilon,k)$ of vector fields $Y\in \Vec^G(M)$ such that
$$
\sup_{x\in \phi(K), 0\leq |\alpha|\leq k} \|D^\alpha (\phi_*X)(x)-D^\alpha (\phi_*Y)(x)\|<\varepsilon.
$$
The $C^\infty$ topology on compact sets is defined as the set of all such sets $\mathcal{U}(X,\phi,K,\varepsilon,k)$. Since no confusion is possible, the quotient topology on $\Vec^G(M_G)$ is also called $C^\infty$ topology on compact sets.

The set $\Vec^G(M)$ is an involutive distribution, therefore it is integrable according to Frobenius's theorem. We foliate $M$ with maximal connected integral manifolds of this distribution. An \emph{equivariant leaf} is any of the leaves of this foliation. As a consequence, it is possible to connect any two points in an equivariant leaf by following equivariant vector fields.

\begin{definition}[Controllability in equivariant leaves]\label{d:continleaves}
Let $X_1,\ldots,X_k\in\Vec^G(M)$. We say that controllability holds  in equivariant leaves if for any $q,q'\in M\setminus M'$ belonging to the same equivariant leaf, there exist $m\in\N$, $t_1,\ldots,t_m\in \R$ and $i_1,\ldots,i_m\in[k]$ (not necessarily distinct) such that 
$$
q'=e^{t_1X_{i_1}}\circ \ldots \circ e^{t_mX_{i_m}}q
$$
where $e^{tX}$ denotes the flow at time $t$ of the vector field $X$ on $M$.
\end{definition}
\begin{theorem}\label{t:precisedstatement}
There exists for any $k\geq 2$ a set of $k$-uples $(X_1,\ldots,X_k)\in (\Vec^G(M))^k$ which is residual in $(\Vec^G(M))^k$ and for which controllability in equivariant leaves holds.
\end{theorem}
We refer the reader to Example \ref{ex:equality} for an illustration of the concept of controllability in equivariant leaves.

Recall that the Lie algebra generated by a family $\mathcal{F}\subset \Vec^G(M)$ is the smallest sub-algebra of $\Vec^G(M)$ containing $\mathcal{F}$, namely 
$$
{\rm Lie}(\mathcal{F}) := {\rm Span}\left\{[Y_1,\ldots, [Y_{j-1}, Y_j]], Y_i \in \mathcal{F}, j\in\N\right\}.
$$
Theorem \ref{t:precisedstatement} is a corollary of the following result (see Section \ref{s:conclupro} for a proof of this implication).
\begin{theorem}\label{t:genliebracket22} 
For any $k\geq 2$, there exists a set of $k$-uples $(X_1,\ldots,X_k)\in (\Vec^G(M))^k$ which is residual in $(\Vec^G(M))^k$ and for which
\begin{equation}\label{e:bracketgenerating22}
\forall q\in M\setminus M', \quad {\rm Lie}_q(X_1,\ldots,X_k)=\Vec^G(M)_{|q}.
\end{equation}
\end{theorem}

We end this section with an example showing that controllability does not necessarily hold inside connected components of $M'$ even if they form equivariant leaves, which explains why the set $M'$ is excluded in Definition \ref{d:continleaves}.
\begin{example}\label{e:pascontM'}
We consider $M=\R^d/\Z^d$ for some $d>2$, and $G$ the set of all translations in $M$. The action is transitive. Then equivariant vector fields are constant vector fields, automatically tangent to the unique orbit, therefore $M'=M$. However, for any $k\leq d-1$ equivariant vector fields $X_1,\ldots,X_k$, controllability cannot hold inside $M'$. Indeed, since $X_1,\ldots,X_k$ are constant, the integral curves of the distribution spanned by $X_1,\ldots,X_k$ are strict subsets of $M$.
\end{example}
\begin{remark}
We will see in Corollary \ref{c:M'/G} that $M'/G$ is a set of isolated points in $M_G$, and that each point in $M'/G$ is a stratum in \eqref{e:stratifM}. Therefore controllability in the sense of Definition \ref{d:continstrata} is automatically satisfied in $M'/G$, although all elements of $\Vec^G(M_G)$ vanish in $M'/G$ by definition of $M'$. This is why the set $M'/G$ does not need to be excluded in the statement of Theorem \ref{t:genequiv}. On the other side, Example \ref{e:pascontM'} shows that controllability does not hold inside $M'$, which illustrates why $M'$ needs to be removed in the statement of Theorem \ref{t:precisedstatement} (actually it is removed in Definition \ref{d:continleaves}).
\end{remark}

\subsection{Averaging}\label{s:prel}
We gather in this section a definition and a proposition that are required for the rest of the proof.

First we define the averaging operator, which turns any vector field on $M$ to an equivariant one by averaging over $G$-orbits. Recall that equivariant vector fields are vector fields on $M$, and that they induce vector fields on $M_G$. 
\begin{definition}[Averaging]
Given $X\in \Vec(M)$, we set 
\begin{equation}\label{e:Xeq}
X^{eq}=\int_G (P_g)_*X \; d\mu_G(g ) \in \Vec^G(M)
\end{equation}
where $\mu_G$ is the normalized Haar measure on $G$.
\end{definition}
Averaging commutes with Lie brackets, as shown in the following proposition in which $X$ is required to be equivariant.
\begin{proposition}[Averaging and bracketing commute]
If $X\in \Vec^G(M)$, then for any $Y\in\Vec(M)$,
\begin{equation}\label{e:commuteeq}
[X,Y^{eq}]=[X,Y]^{eq}
\end{equation}
\end{proposition}
\begin{proof}
Since $X$ is equivariant, there holds $(P_g)_*X=X$ for any $g\in G$, hence
$$
[X,Y^{eq}]=\int_G[X,(P_g)_*Y] \; d\mu_G(g )=\int_G[(P_g)_*X,(P_g)_*Y] \; d\mu_G(g )=[X,Y]^{eq}
$$
where for the last equality we used $(P_g)_*[X,Y]=[(P_g)_*X,(P_g)_*Y]$.
\end{proof}

\subsection{Density of vector fields transverse to $G$-orbits}\label{s:densitytransverse}

We introduce a stratification of $M$, whose link with the stratification \eqref{e:stratifM} is clarified below. The action of $G$ on $M$ induces a natural stratification
\begin{equation}\label{e:stratifMnonquotient}
M=\bigsqcup_{i\in\I} S^i
\end{equation}
where for each $i\in\I$, $S^i$ is a connected component of the set of points of $M$ whose isotropy group is conjugate to some given subgroup $H_i$ of $G$.  The slice theorem recalled in the proof of Lemma \ref{l:fundamental} has several consequences. Firstly, according to \cite[Proposition 3.7.2]{fieldbook}, each $S^i$ is a smooth submanifold. Secondly, $\I$ is countable: this is e.g. a consequence of the proof of \cite[Proposition 3.7.4]{fieldbook}. Moreover, as a consequence of Lemma \ref{l:fundamental} below, any equivariant leaf (see Section \ref{s:stronger}) is contained in one of the strata in \eqref{e:stratifMnonquotient}. The image of any stratum in \eqref{e:stratifMnonquotient} under quotient by $G$ is a stratum in \eqref{e:stratifM}. Conversely, the preimage of any stratum in \eqref{e:stratifM} is the union of a finite number of strata in \eqref{e:stratifMnonquotient}.

Recall that $M$ is assumed real-analytic. For any invariant subset $U\subset M$, we denote by $\mathcal{X}_{\rm eq}^\omega(U)$ (resp. $\mathcal{X}_{\rm eq}^\infty(U)$) the set of real-analytic (resp. smooth) equivariant vector fields on $U$ and by $C^\infty_G(U)$ the set of smooth $G$-invariant real-valued functions on $U$. We consider the $C^\infty_G(U)$-module 
$$
\mathcal{A}_U=\left\{ \sum_{j\in\mathcal{J}} a_jY_j \mid \mathcal{J} \text{ finite}, \; \forall j\in\mathcal{J}, a_j \in C_G^\infty(U), \ Y_j\in \mathcal{X}_{\rm eq}^\omega(U)\right\}.
$$
Then we set 
$$
\mathcal{A}=\{X\in \mathcal{X}^\infty_{\rm eq} (M) \mid \forall x, \exists U_x \neq \emptyset \text{ invariant open neighborhood of $x$ s.t. } X_{|U_x}\in\mathcal{A}_{U_x}\}
$$

For $x\in M$, we denote by $A_x$ the operator $A_x:\mathcal{A}\rightarrow T_xM$,
\begin{equation}\label{e:averagingop}
A_x:X\mapsto X^{eq}(x)
\end{equation}
(see \eqref{e:Xeq}). Its image is denoted by ${\rm Im}(A_x)$. We recall that $G$-orbits, i.e., sets of the form $Gx$ for some $x\in M$, are closed submanifolds (see e.g. \cite[Corollary 3.1.17]{fieldbook}). 
\begin{lemma}\label{l:fundamental}
Let $S$ be one of the strata in \eqref{e:stratifMnonquotient} and let $x\in S$. Then
\begin{equation}\label{e:eqmodulo}
\emph{Im}(A_x)\subset T_xS \subset \emph{Im}(A_x)+T_x(Gx).
\end{equation}
\end{lemma}
 Lemma \ref{l:fundamental} is fundamental in the sequel, and it may be tested for instance on Example \ref{ex:equality} for which the inclusion of $\text{Im}(A_x)$ in $T_xS$ is strict. 
\begin{proof}[Proof of Lemma \ref{l:fundamental}]
Our proof of \eqref{e:eqmodulo} is based on the slice theorem (Theorem 3.5.2 in \cite{fieldbook}) whose statement is the following: for every $x\in M$, it is possible to choose a smooth family of pairwise-disjoint slices 
$$
\mathcal{S} = \{S_y \mid y \in Gx\}
$$ 
satisfying the following properties:
\begin{itemize}
\item For $y\in Gx$, $S_y\subset M$ is a $G_y$-invariant embedded disk of $M$ of dimension $\dim(M)-\dim(Gx)$ which is transverse to $Gx$.
\item For $y\in Gx$ and $g\in G$, $g\cdot S_y=S_{g\cdot y}$. In particular, $g\cdot S_{y}=S_y$ for $g\in G_y$.
\item For $y\in Gx$, $S_y$ is $G$-equivariantly diffeomorphic via $\exp_y$ to the representation given by the linear action of $G_y$ on $(T_yGx)^\perp$.
\item For $z\in S_y$, $G_z$ is a subgroup of $G_y$ (see Lemma 3.7.1(a) in \cite{fieldbook}).
\item $GS_x = \bigcup_{y\in Gx} S_y$ is an open $G$-invariant neighbourhood of $Gx$ which is $G$-equivariantly diffeomorphic to the twisted product $G \times_{G_x} S_x$.
\end{itemize}
These slices are actually defined as follows. We first define a $G$-invariant metric $\nu$ on $M$ by taking any Riemannian metric on $M$, then pushing it forward by the $G$-action and finally averaging the result with respect to the Haar measure. For $y\in Gx$ we denote by  $N_y(\varepsilon)$ the set of $v\in(T_yGx)^\perp$ such that $\nu_y(v)<\varepsilon$. Then the slices $S_y$ for $y\in Gx$ are given by $S_y=\exp_y(N_y(\varepsilon))$ for some $\varepsilon$ sufficiently small and depending only on $x$. The key point is that the exponential map $\exp:TM\rightarrow M$ is $G$-equivariant, giving the linear action of $G_y$ on $(T_yGx)^\perp$. It is easily checked that if $z=\exp_y(v)$ for some $v\in N_y(\varepsilon)$, then 
\begin{equation}\label{e:subgroup}
G_z=\{g\in G_y \mid g_*(v)=v\}.
\end{equation}
 Therefore, the slice theorem  follows from this construction.

We first prove that
\begin{equation}\label{e:equivariantconstantisotropy}
\text{Im}(A_x)=T_xE_x
\end{equation}
where $E_x$ is the set of points with same isotropy group as $x$. Notice that $E_x$ is locally diffeomorphic near $x$ to the subspace of $(T_xGx)^\perp$ given by vectors which are invariant under the linear action of $G_x$ (by \eqref{e:subgroup}), hence it is a submanifold. Also, this observation shows that $\text{Im}(A_x)\subset T_xE_x$ since equivariant vector fields evaluated at $x$ are invariant under $G_x$.

We then prove $T_xE_x\subset \text{Im}(A_x)$. Let $v\in T_xE_x$, then $h_*v=v$ for any $h\in G_x$. We define a vector field on $S_x$ as follows, using the third point of the slice theorem: we consider the preimage $u$ of $v$ through the diffeomorphism $\exp_x$, which is a vector at the origin in $(T_xGx)^\perp$  which verifies $g_*(u)=u$ for any $g\in G_x$, we extend $u$ to a constant vector field on $(T_xGx)^\perp$, and then we push it forward to $S_x$ through $\exp_x$. This vector field is well-defined and equivariant because if $g_2z=g_1z$ for some $z\in S_x$, then $g_2=g_1h$ for some $h\in G_x$ according to the second point of the slice theorem. We extend this vector field to $GS_x$ by pushforward by $G$, obtaining an equivariant and analytic (because the action is analytic) vector field. We then use a smooth $G$-invariant cut-off\footnote{this is the main reason why we introduced the module $\mathcal{A}$: to be able to make cut-offs of analytic fields.} equal to $1$ on $GS_x$ to obtain an element $X\in\mathcal{A}$, equal to $v$ on $GS_x$. The involution $g\mapsto g^{-1}$ preserves the normalized Haar measure $\mu_G$ on $G$ because compact Lie groups are unimodular (i.e. the left-invariant measure is also right-invariant).
Hence
\begin{equation*}
X^{eq}(x)=\int_G((P_g)_*X)(x)\,d\mu_G(g)=\int_G((P_{g^{-1}})_*X)(x)\,d\mu_G(g)=   \int_G(P_{g^{-1}})_*(P_{g})_*v\,d\mu_G(g)=v
\end{equation*}
therefore $v\in\text{Im}(A_x)$ which concludes the proof of \eqref{e:equivariantconstantisotropy}.   

To prove \eqref{e:eqmodulo} it is now sufficient to prove $T_xE_x\subset T_xS \subset T_xE_x+T_x(Gx)$. The inclusion $T_xE_x\subset T_xS$ follows from the definition of strata. Let us prove that $T_xS\subset T_xE_x+T_x(Gx)$. For this, we use the fourth point of the slice theorem. Since all elements of $S$ have isotropy groups conjugated to $G_x$, we deduce $S_x\cap S\subset E_x$. Hence 
$$
T_xS\subset T_x(S_x\cap S)+T_x(Gx)\subset T_x E_x+T_x(Gx),
$$
which concludes the proof of \eqref{e:eqmodulo}.
\end{proof}
The following corollary may be deduced from the above proof.
\begin{corollary}\label{c:M'/G}
$M'/G$ is a set of isolated points, and each point of $M'/G$ is a stratum in \eqref{e:stratifM}.
\end{corollary}
\begin{proof}
Let $x\in M'$, and denote by $S$ the stratum containing $x$. If the smooth submanifold $Gx$ has dimension $0$, then according to Lemma \ref{l:fundamental} there holds $T_xS=\{0\}$, therefore $Gx$ is a stratum in \eqref{e:stratifM}, reduced to a point. If $Gx$ has dimension $\geq 1$, since $Gx\subset S$ and $x\in M'$, Lemma \ref{l:fundamental} implies that $T_xS=T_x(Gx)$. The slice theorem shows that if $z\in S_x\setminus \{x\}$ (where $S_x$ denotes the slice at $x$, see proof of Lemma \ref{l:fundamental}), then $G_z\subsetneq G_x$. Therefore $Gx\in M'/G$ is a stratum in \eqref{e:stratifM}.

As recalled at the beginning of Section \ref{s:densitytransverse}, the number of strata is locally finite. Therefore, $M'/G$ is a set of isolated points.
\end{proof}

We say that $X\in \Vec(M)$ is transverse to the $G$-orbit at $x$ if $X(x)\notin T_x(Gx)$. Lemma \ref{l:fundamental} is useful to prove the following result.
\begin{lemma}\label{l:septransverse}
Assume $k\geq 2$.
Then there exists a residual set 
of $k$-uples $(X_1,\ldots,X_k)\in \mathcal{A}^k$ such that for any $x\in M\setminus M'$, at least one of the vectors $X_1(x),\ldots,X_k(x)$ is transverse to the $G$-orbit at $x$. 
\end{lemma}
\begin{proof}
In the proof of Lemma \ref{l:fundamental} we recalled the slice theorem. We also recall from \cite[Proposition 3.7.4]{fieldbook} that if $M$ is a compact $G$-manifold or a $G$-representation, then the number of isotropy types (i.e. different isotropy subgroups, up to conjugation) for the $G$-action is finite. It follows from the proof of  \cite[Proposition 3.7.4]{fieldbook} that even if $M$ is not assumed compact, for any compact set $K\subset M$, the number of strata of $M$ in the sense of the stratification \eqref{e:stratifMnonquotient} (in particular strata are connected sets) which have non-empty intersection with $K$ is finite. This property is called Property \textbf{P} in the sequel.

We exhaust the open set $M\setminus M'$ by increasing compact sets $M_j$, $j\in\N$ (this is possible thanks to Whitney's embedding theorem) assumed to be $G$-invariant: 
$$
M=\bigcup_{j\in\mathbb{N}} M_j, \qquad \forall j\in\mathbb{N}, \, M_j\subset M_{j+1}, \qquad GM_j=M_j.
$$
In the sequel $j\in\N$ is fixed. Let $S$ be a stratum intersecting $M_j$. The set $S/G$ is a smooth manifold, and we denote its dimension by $\ell(S)$. The dimension of $T_x(Gx)$ does not depend on $x\in S$ (because $G_x$, $G_{x'}$ are conjugate for $x,x'\in S$) and there holds $\ell(S)=\text{dim}(T_xS/T_x(Gx))$ for any $x\in S$. 

We consider for $x\in M$ the linear map $h_x:\mathcal{A}^k\rightarrow (T_xM/T_x(Gx))^k$ 
$$
h_x:(X_1,\ldots,X_k)\mapsto (A_xX_1\ {\rm mod}(T_x(Gx)),\ldots,A_xX_k\ {\rm mod}(T_x(Gx))).
$$
According to Lemma \ref{l:fundamental}, for $x\in S$, the application $h_x$ has rank $k\ell(S)$. The domain of $h_x$ is infinite-dimensional, but to compute codimensions, we restrict $h_x$ to a finite-dimensional space, while preserving its range: for any $x\in M_j$, we choose a finite dimensional subspace of $\mathcal{A}^k$ such that for any $y$ in some open neighborhood of $x$ the restriction of $h_y$ to this subspace has same range as $h_y$. Covering the compact set $M_j$ with a finite number of such open neighborhoods, we end-up with a finite dimensional subspace $F_j\subset \mathcal{A}$ such that the restriction $\tilde{h}_x:=h_{x|F_j^k}$ has also rank $k\ell(S)$ for any $x\in M_j\cap S$. Its kernel $\tilde{h}_x^{-1}(0)\subset F_j^k$ has codimension $k\ell(S)$ for $x\in M_j\cap S$. The union $\bigcup_{x\in M_j\cap S} \tilde{h}_x^{-1}(0)$ is a subset of $F_j^k$ of codimension $\geq(k-1)\ell(S)$. This quantity is $\geq 1$ as soon as $\ell(S)\geq 1$ (since $k\geq 2$).

The above reasoning implies that for any $j\in\N$, for any stratum $S$ having non-empty intersection with $M_j$, $\bigcup_{x\in M_j\cap S} h_x^{-1}(0)$ is a subset of codimension $\geq (k-1)\ell(S)$. Taking the union over the locally finite number (according to Property \textbf{P} above) of strata for which $\ell(S\cap M_j)\neq 0$, we obtain that the codimension of $\bigcup_{x\in M_j} h_x^{-1}(0)$ is $\geq 1$. Therefore there exists a dense set of $k$-uples $(X_1,\ldots,X_k)\in\mathcal{A}^k$ on $M$ such that for any $x\in M_j$ there exists $i\in\{1,\ldots,k\}$ having the property that $X_i(x)$ is transverse to the $G$-orbit at $x$. This set of $k$-uples is open since $M_j$ is compact. Taking the intersection of these sets over $j\in\mathbb{N}$, we obtain Lemma \ref{l:septransverse}.
\end{proof}
\begin{remark}
Assume $G$ is finite. Since any orbit is discrete, a vector field $X$ is transverse to the $G$-orbit at $x$ if and only if $X(x)\neq 0$. Therefore, Lemma \ref{l:septransverse} means that for an open dense set of $k$-uples $(X_1,\ldots,X_k)$ of elements of $\mathcal{A}^k$, for any $x\in M\setminus M'$ at least one of the vectors $X_1(x),\ldots,X_k(x)$ is $\neq 0$. 
\end{remark}

\subsection{Proof of Theorem \ref{t:genliebracket22}}
Recall the notation 
$$
{\rm ad}_X^0Y=Y, \qquad {\rm ad}_X^kY=[X,{\rm ad}_X^{k-1}Y], 
$$ 
for any $k\geq 1$ and any vector fields $X,Y$.
Our proof of Theorem \ref{t:genliebracket22} is based on the following lemma.
\begin{lemma}\label{l:inoneOi}
Let $\mathcal{K}\subset M$ be a $G$-invariant compact set. Let $X,Y\in\mathcal{A}$ and let $\mathcal{O}\subset \mathcal{K}$ be a tubular open set of the form
\begin{equation} \label{e:U}
\mathcal{O}=\bigsqcup_{t\in ]-T,T[} e^{tX}\Sigma
\end{equation}
where $\Sigma\subset M$ is a $G$-invariant hypersurface, $X$ is transverse to $\Sigma$, and $T>0$ is small enough so that \eqref{e:U} defines tubular coordinates in $\mathcal{O}$. There exists $p(\mathcal{K})\in\N$ (depending only on $\mathcal{K}$) such that for any $p\geq p(\mathcal{K})$, any $\varepsilon>0$ and any neighborhood of the closure $\overline{\mathcal{O}}$, there exists $Z\in\mathcal{A}$ supported in this neighborhood and with $\|Z\|_{C^p}\leq \varepsilon$ such that 
$$
\forall q\in\mathcal{O}, \quad \emph{Span}\left({\rm ad}_X^0(Y+Z)(q),\ldots,{\rm ad}^p_X(Y+Z)(q)\right)= \Vec^G(\mathcal{O})_{|q}.
$$
\end{lemma}

We postpone the proof of Lemma \ref{l:inoneOi} to the end of this section, and first explain how to finish the proof of Theorem \ref{t:genliebracket22}. In the $C^\infty$ topology on compact sets (defined in Section \ref{s:stronger}), $k$-uples of vector fields satisfying \eqref{e:bracketgenerating22} form a countable intersection of open sets. Therefore, we only need to prove their density.

Let us fix $(X_1,\ldots,X_k)\in (\Vec^G(M))^k$. By density of $\mathcal{A}$ in $\Vec^G(M)$ together with Lemma \ref{l:septransverse}, we may assume that $(X_1,\ldots,X_k)\in\mathcal{A}^k$ satisfy the conclusion of Lemma \ref{l:septransverse}, i.e., for any $x\in M\setminus M'$, at least one of the vectors $X_1(x),\ldots,X_k(x)$ is transverse to the $G$-orbit at $x$. Thanks to Lemma \ref{l:septransverse}, we cover $M\setminus M'$ with $G$-invariant open sets $(\mathcal{O}_j)_{j\in\J}$ which are tubular neighborhoods of the form \eqref{e:U} for some hypersurface $\Sigma:=\Sigma_j$ transverse to the $G$-orbits in $\mathcal{O}_j$, some $T:=T_j\in\R^+$ and some $X:=X_{i_j}$ where $i_j\in\{1,\ldots,k\}$ for any $j\in\J$. We may assume that this covering is locally finite, i.e., for any compact subset of $M\setminus M'$, the number of elements of this covering which intersect this compact set is finite, therefore $\mathcal{J}=\N$ or $\mathcal{J}=\{1,\ldots,J\}$ for some $J\in\N$. We also fix an increasing sequence $(\mathcal{K}_\ell)_{\ell\in\N}$ of compact sets such that $M=\bigcup_{\ell\in\N} \mathcal{K}_\ell$.

We modify the vector fields $X_1,\ldots,X_k$ inductively, for $j=1,2,\ldots$ ($j\in\mathcal{J}$). At the end of step $j\in\mathcal{J}$, we have the vector fields $X_1^{(j)},\ldots,X_k^{(j)}$ and we ensure that:
\begin{enumerate}[(1)]
\item $X_1^{(j)},\ldots,X_k^{(j)}$ satisfy 
$$
\forall q\in \bigcup_{j'\leq j}\mathcal{O}_{j'}, \qquad {\rm Lie}_q(X_1,\ldots,X_k)=\Vec^G(M)_{|q}.
$$
\item For any $j'\in\mathcal{J}$, $X^{(j)}_{i_{j'}}$ is transverse to $\Sigma_{j'}$. 
\end{enumerate}

Fix $j\in\mathcal{J}$ and assume that step $j-1$ has been done. We pick $\alpha_j\in [k]\setminus \{i_j\}$ (arbitrarily). We modify $X_{\alpha_j}^{(j-1)}$ in a neighborhood of $\mathcal{O}_{j}$. The modification only affects $X_{\alpha_j}$: it is of the form 
\begin{align}
X^{(j)}_{\alpha_j}&=X^{(j-1)}_{\alpha_j}+\varphi_j Z_j \label{e:defX_j}\\
X^{(j)}_i&=X^{(j-1)}_i \text{ for } i\neq \alpha_j.\nonumber
\end{align}
In particular $X_{i_j}^{(j)}=X_{i_j}^{(j-1)}$. Let us explain the construction of $\varphi_j$ and $Z_j$.

Let us fix an index $\ell$ such that $\mathcal{O}_j\subset\mathcal{K}_\ell$. We choose $Z_j$ thanks to Lemma \ref{l:inoneOi} such that for any $q\in\mathcal{O}_{j}$ the vectors 
$$
ad_{X^{(j)}_{i_j}}^k X_{\alpha_j}^{(j)}(q)=ad_{X^{(j-1)}_{i_j}}^k(X^{(j-1)}_{\alpha_j}+Z_j)(q), \qquad k=0,\ldots,p
$$ 
span $\Vec^G(M)_{|q}$. The regularity index $p$ above only depends only on $\mathcal{K}_\ell$ but not on $\mathcal{O}_j$ according to Lemma \ref{l:inoneOi}. We also take $\varphi_j$ a $C^\infty(M)$ and $G$-invariant cutoff function supported near $\mathcal{O}_j$, with value $1$ in $\mathcal{O}_j$, and $0$ outside a small neighborhood of $\mathcal{O}_j$.   Moreover, we require the following properties:
\begin{itemize}
\item $\|Z_j\|_{C^\reg}\leq \varepsilon 2^{-j}$. 
\item (1) and (2) are satisfied.  
\end{itemize}
The second bullet is guaranteed by taking $\|Z_j\|_{C^\reg}$ sufficiently small and the support of $\varphi_j$ to be a sufficiently small neighborhood of the closure $\overline{\mathcal{O}}_j$. Here we use the fact that the covering of $M\setminus M'$ is locally finite, hence the transversality condition (2), which is an open condition, is perturbed only for a finite number of $j'\in \J$, and therefore remains true if the perturbation is sufficiently small.

Once $j$ has run over $\mathcal{J}$, and at each step a perturbation of the form \eqref{e:defX_j} has been added, we obtain modified vector fields which we denote by $X_1',\ldots,X_k'\in {\rm Vec}^G(M)$. Convergence of the series of modifications is guaranteed by the first bullet above, with $\|X_i'-X_i\|_{C^{\reg_\ell}(\mathcal{K}_\ell)}\leq \varepsilon$ for any $i\in[k]$ and for some $p_\ell$ depending only on $\mathcal{K}_\ell$. Moreover, the vector fields satisfy \eqref{e:bracketgenerating22} thanks to (1). This concludes the proof of Theorem \ref{t:genliebracket22}.

We finally prove  Lemma \ref{l:inoneOi}.
\begin{proof}[Proof of Lemma \ref{l:inoneOi}]
Any module generated by real-analytic vector fields is locally finitely generated, due to the Nötherian property of the ring of germs of real-analytic functions (see \cite[Theorem I.9]{frisch}). Therefore, the module $\mathcal{A}_{\mathcal{K}}$ is locally finitely generated: there exist $m\in\N$ and analytic vector fields $f_1,\ldots,f_m$ on $\mathcal{K}$ such that 
\begin{equation}\label{e:finitelygenerated}
\mathcal{A}_{\mathcal{K}}=\left\{ \sum_{i=1}^m a_if_i \mid \forall i\in[m],\ a_i \in C_G^\infty(\mathcal{K})\right\}.
\end{equation}
Of course, the analytic vector fields $f_1,\ldots,f_m$, when restricted to $\mathcal{O}$, also generate $\mathcal{A}_O$. But it is important for us that the number $m$ of vector fields depends only on $\mathcal{K}$ (not on $\mathcal{O}\subset \mathcal{K}$).
Besides, $f_1,\ldots,f_m\in\mathcal{X}_{eq}^\omega(\mathcal{K})$ since all vector fields in $\mathcal{A}_{\mathcal{K}}$ are equivariant.

Since $\text{ad}_Xf_i\in\mathcal{A}_{\mathcal{O}}$ (according to \eqref{e:commuteeq})  and $\mathcal{A}_{\mathcal{O}}$ is finitely generated by the $f_j$,
$$
\text{ad}_Xf_i=\sum_{j=1}^m a_{ij}f_j
$$
for some $a_{ij}\in C_G^\infty(\mathcal{O})$, $1\leq i,j\leq m$. We set 
\begin{equation}\label{e:deffit}
f_i^t:= (e^{tX})_* f_i
\end{equation}
and observe that
\begin{equation}\label{e:EDO}
\frac{d}{dt}f_i^t=-\sum_{j=1}^m (a_{ij}\circ e^{-tX})f_j^t, \qquad f_i^0=f_i.
\end{equation}
Since the equation \eqref{e:EDO} is linear and the $f_j$ verify \eqref{e:finitelygenerated}, the $f_j^t$, $j=1,\ldots,m$ verify
\begin{equation}\label{e:Aatt}
\mathcal{A}_{\mathcal{O}}=\left\{ \sum_{i=1}^m a_if_i^t \mid \forall i\in[m],\ a_i \in C_G^\infty(\mathcal{O})\right\}
\end{equation}
for any $t\in ]-T,T[$. For some sufficiently large $p\in\mathbb{N}$ and some matrix $\alpha\in \mathcal{M}_{m\times (p+1)}$, both to be chosen later, we consider for $i=0,\ldots,p$
\begin{equation}\label{e:gialphafj}
g_i=\sum_{j=1}^m \alpha_{ij}f_j.
\end{equation}
Then $g_i^t:= (e^{tX})_* g_i$ verify
\begin{equation}\label{e:gitfit}
g_i^t=\sum_{j=1}^m \alpha_{ij}f_j^t.
\end{equation}
In the sequel, each point of $\mathcal{O}$ is written as $e^{tX}q_0$ where $t\in ]-T,T[$ and $q_0\in \Sigma$, thanks to \eqref{e:U}.
For such $g_i$'s, we consider
\begin{equation}\label{e:Zfunctionalpha}
Z(e^{tX}q_0)=(e^{tX})_*\left(\sum_{i=0}^p \frac{t^i}{i!}g_i(q_0)\right).
\end{equation}
We want to compute $\text{ad}_X^\gamma Z$ for $\gamma=0,\ldots,p$. For this, we notice 
$$
 [(e^{\varepsilon X})_* Z](e^{tX}q_0)=(e^{\varepsilon X})_*(e^{(t-\varepsilon)X})_*\left(\sum_{i=0}^p \frac{(t-\varepsilon)^i}{i!}g_i(q_0)\right)=\sum_{i=0}^p \frac{(t-\varepsilon)^i}{i!}g_i^t(e^{tX}q_0)
$$
where for the second equality we used $g_i^t=(e^{tX})_*g_i$.
We differentiate $\gamma$ times with respect to $\varepsilon$ at $0$: we obtain
$$
[\text{ad}_X^\gamma Z](e^{tX}q_0)=(-1)^\gamma\frac{d^\gamma}{d\varepsilon^\gamma}_{|\varepsilon=0} \sum_{i=0}^p \frac{(t-\varepsilon)^i}{i!}g_i^t(e^{tX}q_0)=\sum_{i=\gamma}^p \frac{t^{i-\gamma}}{(i-\gamma)!}g_i^t(e^{tX}q_0).
$$
Therefore writing $\text{ad}_X^\gamma Y=\sum_{j=1}^m \beta_{\gamma j}f_j$ where $\beta=(\beta_{kj})$ is a $(p+1)\times m$ matrix, we obtain
\begin{equation}\label{e:adXkYZ}
\text{ad}_X^\gamma (Y+Z)(e^{tX}q_0)= \sum_{j=1}^m \underbrace{\left(\beta_{\gamma j}(e^{tX}q_0)+ \sum_{i=\gamma}^p \frac{t^{i-\gamma}}{(i-\gamma)!} \alpha_{ij}\right)}_{:=\delta_{\gamma j}} f_j^t(e^{tX}q_0)
\end{equation}
thanks to \eqref{e:gitfit}.
The goal is to choose the constant-coefficients matrix $\alpha=(\alpha_{ij})$ (with $0\leq \gamma\leq p$ and $1\leq j\leq m$) in a way that $\delta=(\delta_{\gamma j})\in \mathcal{M}_{(p+1)\times m}$ defined in \eqref{e:adXkYZ} has rank $m$ at any point in $\mathcal{O}$. We notice that $\eta(t)\in \mathcal{M}_{(p+1)\times (p+1)}$ defined by its coefficients
\begin{equation}\label{e:eta}
\eta_{\gamma i}(t)=\mathbf{1}_{i\geq \gamma} \frac{t^{i-\gamma}}{(i-\gamma)!}
\end{equation}
is a triangular matrix with non-zero diagonal coefficients, hence it is invertible. Thus, $\delta=\beta+\eta\alpha$ has rank $m$ if and only if $\eta^{-1}\beta +\alpha$ has rank $m$. When $x=e^{tX}q_0$ varies in $\mathcal{O}$, $\eta(t)^{-1}\beta(e^{tX}q_0)$ describes a submanifold of $\mathcal{M}_{(p+1)\times m}$ of dimension $\leq n=\text{dim}(M)$. Hence, if $p+1\geq n+m$, then for $\alpha$ in a codimension $p+2-m-n\geq 1$ submanifold, $\delta$ has rank $m$ at any point in the neighborhood.
 
Fix such $\alpha$ and take $Z$ according to \eqref{e:gialphafj}, \eqref{e:Zfunctionalpha}. As a consequence of \eqref{e:Aatt}, \eqref{e:adXkYZ} and the fact that $\delta$ has rank $m$, we have at any point $q\in\mathcal{O}$
$$
\text{Span}\left(\text{ad}_X^0(Y+Z)(q),\ldots,\text{ad}^p_X(Y+Z)(q)\right)=\text{Span}(f_1^t(q),\ldots,f_m^t(q))= \Vec^G(\mathcal{O})_{|q}.
$$
Finally, recall that $m$ depends only on $\mathcal{K}$, therefore $p(\mathcal{K}):=n+m-1$ depends only on $\mathcal{K}$, which concludes the proof.
\end{proof}

\subsection{Proof of Theorem \ref{t:precisedstatement} and an example} \label{s:conclupro}
Each equivariant leaf is by definition a submanifold of $M$. Let $X_1,\ldots,X_k$ be a tuple satisfying the conclusion of Theorem \ref{t:genliebracket22}. We may apply Chow-Rashevskii's theorem \cite[Theorems 5.1 and 5.2]{agrachevsachkov} in any equivariant leaf $L$. We obtain that for any $q,q'\in L$, there exist $m\in\N$, $t_1,\ldots,t_m\in \R$ and $i_1,\ldots,i_m\in[k]$ (not necessarily distinct) such that $q'=e^{t_1X_{i_1}}\circ \ldots \circ e^{t_mX_{i_m}} q$. In other words, controllability in equivariant leaves holds, which concludes the proof of Theorem \ref{t:precisedstatement}.

The following example illustrates the fact that controllability does not necessarily generically hold in strata of $M$ given in \eqref{e:stratifMnonquotient}. This is because strata are possibly larger sets than equivariant leaves.
\begin{example}\label{ex:equality}
Let $m,\ell\geq 1$ and take $M=\mathbb{S}^m\times\mathbb R^\ell$ (or $M=\mathbb{S}^m\times\mathbb{T}^\ell$ with $\mathbb{T}=\R/\mathbb{Z}$ if we want $M$ to be compact) equipped with the action of the orthogonal group $G=O(m+1)$ given by $O\cdot (a,b)=(Oa,b)$ for $(a,b)\in \mathbb{S}^m\times\mathbb R^\ell$. The equivariant vector fields are all vector fields that are tangent to $\mathbb{R}^\ell$ and do not depend on the point on the sphere. There is a single stratum for \eqref{e:stratifMnonquotient}, equal to $M$, and controllability cannot hold in the whole stratum. Equivariant leaves are sets of the form $L_a=\{(a,x) \mid x\in \R^\ell\}$ for $a\in \mathbb{S}^m$, and controllability in equivariant leaves is possible (and generic according to Theorem \ref{t:precisedstatement}). Also, controllability in $M/G=\mathbb{R}^\ell$ holds since $C^\infty$ equivariant vector fields induce on $M/G\simeq \mathbb{R}^\ell$ all smooth fields.
\end{example}

\subsection{Proof of Theorem \ref{t:genequiv}}
Recall that the image of any stratum in \eqref{e:stratifMnonquotient} under quotient by $G$ is a stratum in \eqref{e:stratifM}. 
Combining Theorem \ref{t:genliebracket22} with Lemma \ref{l:fundamental} and Corollary \ref{c:M'/G}, we obtain: 
\begin{theorem}\label{t:genliebracket}
For any $k\geq 2$, there exists a set of $k$-uples $(X_1,\ldots,X_k)\in (\Vec^G(M_G))^k$ which is residual in $(\Vec^G(M))^k$ and for which
\begin{equation}\label{e:bracketgenerating}
\forall q\in M_G, \quad {\rm Lie}_q(X_1,\ldots,X_k)=T_qS_G^i.
\end{equation}
Here $i=i(q)$ denotes the index of the stratum containing $q$.
\end{theorem}
Theorem \ref{t:genequiv} follows from Theorem \ref{t:genliebracket} and Chow-Rashevskii's theorem.

\section{Proof of Theorem \ref{t:Nfold}}\label{s:ensembles}
This section is devoted to the proof of Theorem \ref{t:Nfold}. Since an intersection of residual sets is still a residual set, it is sufficient to prove the result for fixed $N$, i.e., simultaneous controllability in strata for any $N$ points with $N$ fixed. Let us fix $N\in\N$. For a vector field $X\in \Vec(M)$, consider its $N$-fold, defined on the product $M^N$ as 
$$
X^N(x_1,\ldots,x_N)=(X(x_1),\ldots,X(x_N)).
$$
For $X,Y\in \Vec(M)$ and $N\geq 1$ we define the Lie bracket of the $N$-folds $X^N,Y^N$ on $M^N$ ``componentwise": $[X^N,Y^N]=[X,Y]^N$ where $[X,Y]$ is the Lie bracket of $X,Y$ on $M$. The same holds for the iterated Lie brackets.

We set
$$
M^{(N)}=\{(q_1,\ldots,q_N)\in (M\setminus M')^N \mid  \forall i\neq j,\ Gq_i\cap Gq_j=\emptyset\}.
$$
Given equivariant vector fields $X_1,\ldots,X_k$ on $M$, we say that their $N$-folds $X_1^N,\ldots,X_k^N$ form a bracket-generating system in equivariant leaves on $U\subset M^{(N)}$ if 
\begin{equation}\label{e:Nfoldbracket}
\forall (q_1,\ldots,q_N) \in U, \quad {\rm Lie}_{(q_1,\ldots,q_N)}(X_1^N,\ldots,X_k^N)=\Vec^G(M)_{|q_1}\times \ldots\times Vec^G(M)_{|q_N}.
\end{equation}
This equality is written for points in $U\subset M^{(N)}$ because it cannot hold in $M^N\setminus M^{(N)}$ due to equivariance. Notice that for $N>1$, \eqref{e:Nfoldbracket} for $U=M^{(N)}$ is strictly stronger than the property \eqref{e:bracketgenerating22} (which corresponds to the case $N=1$).

In this section we prove the following statement:
\begin{theorem}\label{t:Nfoldbracketgenerating} 
For any $N\geq 1$, there is a residual set of $k$-uples of vector fields $(X_1,\ldots,X_k)$ in $(\Vec^G(M))^k$, such that for any $(X_1,\ldots,X_k)$ from this set the $N$-folds $(X_1^N,\ldots,X_k^N)$ form a bracket generating system in equivariant leaves on $M^{(N)}$.
\end{theorem}
Recall that the image of any stratum in \eqref{e:stratifMnonquotient} under quotient by $G$ is a stratum in \eqref{e:stratifM}. 
Combining Theorem \ref{t:Nfoldbracketgenerating} with Lemma \ref{l:fundamental} and Corollary \ref{c:M'/G}, we obtain the following result (where the definition of $N$-fold of elements of $\Vec^G(M_G)$ is deduced from $N$-fold of elements of $\Vec^G(M)$).
\begin{theorem}\label{t:genliebracketbli}
For any $k\geq 2$, there exists a set of $k$-uples $(X_1,\ldots,X_k)\in (\Vec^G(M_G))^k$ which is residual in $(\Vec^G(M_G))^k$ and for which for any distinct $q_1,\ldots,q_N\in M_G$,
\begin{equation}\label{e:bracketgenerating}
{\rm Lie}_{(q_1,\ldots,q_N)}(X_1^N,\ldots,X_k^N)=T_{q_1}S_G^{i(q_1)}\times \ldots \times T_{q_N}S_G^{i(q_N)}.
\end{equation}
Here $i(q)$ denotes the index of the stratum in \eqref{e:stratifM} containing $q\in M_G$. 
\end{theorem}
Theorem \ref{t:Nfold} is a direct consequence of Theorem \ref{t:genliebracketbli} combined with the Chow-Rashevskii theorem, as in \cite[Proposition 3.1]{agrachevsarychev}. The condition in Definition \ref{d:continstrataensemble} that the stratum has dimension $\geq 2$ is necessary because in dimension $1$, ordering of points is preserved (and indeed, $q_1,\ldots,q_N$ are assumed distinct in \eqref{e:bracketgenerating}, therefore they cannot cross).

The rest of this section is devoted to the proof of Theorem \ref{t:Nfoldbracketgenerating}. Since the proof consists in a slight modification of the proof of Theorem \ref{t:genliebracket22}, we only provide the key ideas and highlight the modifications compared to the proof of Theorem \ref{t:genliebracket22}. 

We write $M^{(N)}$ as a union of compact sets $M^{(N)}_s$, $s\in\N$, which are invariant under the action of $G^N$ given by
$$
(g_1,\ldots,g_N)\cdot (q_1,\ldots,q_N)=(g_1q_1,\ldots,g_Nq_N).
$$
It is sufficient to prove Theorem \ref{t:Nfoldbracketgenerating} in $M^{(N)}_s$ for fixed $s\in\N$, instead of $M^{(N)}$, since taking intersection over $s$ of the sets of vector fields will yield a countable intersection of residual sets, which is still a residual set. Therefore, we fix $s\in\N$ and we prove that there exists $p_s\in\N$ such that 
\begin{center}
\textit{ the set of $k$-uples of vector fields $(X_1,\ldots,X_k)$ in $(\Vec^G(M))^k$, such that \\ the length $\leq p_s$ brackets of the $N$-folds $(X_1^N,\ldots,X_k^N)$ generate in $M_s^{(N)}$ \\  all equivariant leaves is open and dense in $(\Vec^G(M))^k$.} 
\end{center}
Moreover, for any compact set $\mathcal{K}\subset M$, the regularity index $p_s$ in the above statement may be taken the same for all $s$ such that $M_s^{(N)}\subset \mathcal{K}$. 

Openness in the above statement is immediate since $M_s^{(N)}$ is compact. Therefore, we only need to prove the density.

Let us fix $(X_1,\ldots,X_k)\in (\Vec^G(M))^k$. By density of $\mathcal{A}$ in $\Vec^G(M)$ together with Lemma \ref{l:septransverse}, we may assume that $(X_1,\ldots,X_k)\in\mathcal{A}^k$ satisfy the conclusion of Lemma \ref{l:septransverse}, i.e., 
\begin{equation}\label{e:transverseatx}
\forall x\in M\setminus M', \exists i\in[k] \text{ such that } X_i(x) \text{ is transverse to the $G$-orbit at $x$.} 
\end{equation}

For each $(q_1,\ldots,q_N)\in M_s^{(N)}$ we do the following. Thanks to  \eqref{e:transverseatx} we fix a linear combination $a_1X_1+\ldots+a_kX_k$ which is transverse to the $G$-orbit at $q_i$ for any $i\in[N]$. Then we consider a product of tubular neighborhoods 
\begin{equation}\label{e:Oiunion}
\mathcal{O}_1\times \ldots \times \mathcal{O}_N\subset M^{(N)}
\end{equation}
where $q_\ell\in\mathcal{O}_\ell$, each $\mathcal{O}_\ell$ is of the form \eqref{e:U}, and the $\mathcal{O}_\ell$'s have empty intersection, which is possible since $q_1,\ldots,q_N$ are pairwise distinct. We assume that the neighborhood \eqref{e:Oiunion} is sufficiently small so that the linear combination $a_1X_1+\ldots+a_kX_k$ is transverse in $\mathcal{O}_1\cup \ldots\cup \mathcal{O}_N$ to the $G$-orbits.

Doing this for any $(q_1,\ldots,q_N)\in M_s^{(N)}$, we have obtained an open covering of the compact set $M_s^{(N)}$, from which we select a finite sub-covering $\widetilde{V}_1,\ldots,\widetilde{V}_{j_0}$. For fixed $j\in[j_0]$, we have by definition
\begin{equation}\label{e:disjointOi}
\widetilde{V}_j= \mathcal{O}_1^{(j)}\times \ldots \times \mathcal{O}_N^{(j)}
\end{equation}
where each $\mathcal{O}_\ell^{(j)}$ is of the form \eqref{e:U} (with $t\in ]-T_\ell^{(j)},T_\ell^{(j)}[$ and $\Sigma_\ell^{(j)}$ denotes the hypersurface), and $\mathcal{O}_\ell^{(j)}$, $\mathcal{O}_{\ell'}^{(j)}$ are separated (``at positive distance") for any $\ell\neq \ell'$. For any $j\in[j_0]$, there exist $a_1^{(j)},\ldots,a_k^{(j)}\in\R$ such that $a_1^{(j)}X_1+\ldots+a_k^{(j)}X_k$ is transverse to the $G$-orbits in 
\begin{equation}\label{e:tildeVj}
V_j=\bigsqcup_{\ell=1}^N  \mathcal{O}_\ell^{(j)}.
\end{equation}
(Notice that the $V_j$'s do not necessarily have empty intersection.)

We perturb $(X_1,\ldots,X_k)$ in $V_1,\ldots,V_{j_0}$ successively; at step $j$ we perturb the vector fields in $V_j$, and the vector fields which we obtain at the end of this step are denoted by $(X_{1}^{(j)},\ldots,X_{k}^{(j)})$. 

We need to modify the construction made in the proof of Lemma \ref{l:inoneOi}. Fix a step $j\in[j_0]$. We set $X=a_1^{(j)}X_1+\ldots+a_k^{(j)}X_k$ and 
$$
Y=\begin{cases}
X_1 \text{ if } (a_1^{(j)},\ldots,a_k^{(j)})\neq (\lambda,0,\ldots,0) \text{  for any $\lambda\in\R$}\\
X_2 \text{ otherwise.}
\end{cases}
$$
We construct below a perturbation $Z$ of $Y$, supported near $V_j$. We use the same idea of Taylor expansion as in \eqref{e:Zfunctionalpha}, just taking larger $p$ in order to generate all directions in the tangent space at any $(e^{t_1X}q_1,\ldots,e^{t_NX}q_N)$.
We introduce the set $\mathcal{C}$ of parameters $(\mathbf{t},\mathbf{q})$ such that $\mathbf{t}=(t_1,\ldots,t_N)$ with $t_\ell\in]-T_\ell^{(j)},T_\ell^{(j)}[$, and $\mathbf{q}=(q_1,\ldots,q_N)\in \Sigma_1^{(j)}\times\ldots\times \Sigma_N^{(j)}$. 

The vector field $Z$ is constructed in the following way. We pick analytic vector fields $f_1,\ldots,f_m$ as in the proof of Lemma \ref{l:inoneOi}, in order for \eqref{e:finitelygenerated} to hold. We follow \eqref{e:deffit} to \eqref{e:gitfit} in all $\mathcal{O}=\mathcal{O}_\ell^{(j)}$, $i=1,\ldots,N$, and define $Z$ through \eqref{e:Zfunctionalpha}: for each $\ell\in[N]$ there is a matrix $\alpha=\alpha^{(\ell)}$, which defines $Z$ in $\mathcal{O}_\ell^{(j)}$. The vector field $Z$ is thus supported near $V_j$. An analogous formula to \eqref{e:adXkYZ} holds:
\begin{equation}\label{e:newperturb}
\text{ad}_X^\gamma (Y+Z)(e^{t_\ell X}q_\ell )= \sum_{r=1}^m \underbrace{\left(\beta_{\gamma r}^{(\ell)}(e^{t_\ell X}q_\ell )+ \sum_{i=\gamma }^p \frac{t^{i-\gamma}}{(i-\gamma)!} \alpha_{ir}^{(\ell)}\right)}_{:=\delta_{\gamma r}^{(\ell)}} f_r^t(e^{t_\ell X}q_\ell)
\end{equation}

The goal is to choose the constant-coefficients matrices $\alpha^{(\ell)}=(\alpha^{(\ell)}_{ir})$ (with $0\leq \gamma\leq p$ and $r\in[m]$, $\ell\in[N]$) in a way that the  $N(p+1)\times Nm$ block matrix
$$
\delta'=\begin{pmatrix}
\delta^{(1)} & \mathbf{0} & \ldots & \mathbf{0} \\
\mathbf{0} & \delta^{(2)} & \ldots & \mathbf{0} \\
\vdots & \vdots & \ddots &\vdots\\
\mathbf{0} & \mathbf{0}& \ldots & \delta^{(N)}
\end{pmatrix}
$$
(whose $N^2$ blocks are of size $(p+1)\times m$) has rank $Nm$ for any $(\mathbf{t},\mathbf{q})\in\mathcal{C}$. 
 We denote by $\eta'$ the $N(p+1)\times N(p+1)$ matrix defined by $N^2$ blocks of size $(p+1)\times (p+1)$, all equal to $0$ except the on-diagonal ones taken equal to the matrix $\eta$ introduced in \eqref{e:eta}. We define similarly $\beta'$ and $\alpha'$ as  $N(p+1)\times Nm$ block matrices having blocks equal to $0$ except $\ell^{\rm th}$ diagonal block equal respectively to $\beta^{(\ell)}$ and $\alpha^{(\ell)}$ which appear in \eqref{e:newperturb}. We obtain $\delta'=\beta'+\eta'\alpha'$.

 Since $\eta'$ is invertible, $\delta'$ has rank $Nm$ if and only if $\eta'^{-1}\beta' +\alpha'$ has rank $Nm$. When $(\mathbf{t},\mathbf{q})$ vary in $\mathcal{C}$, $\eta'^{-1}\beta'$ describes a submanifold of $\mathcal{M}_{(p+1)N\times mN}$ of dimension $\leq nN$, where $n=\text{dim}(M)$. Hence, if $p+1\geq (m+n)N$, then for $\alpha$ in a codimension $p+2-(m+n)N\geq 1$ submanifold, $\delta'$ has rank $mN$ at any point in the neighborhood.

For $\alpha'$ and $Z$ taken in this way, we obtain that the Lie algebra generated by the (restrictions of the) $N$-folds vector fields $X_1^N,\ldots,X_k^N$ equals $\Vec^G(\mathcal{O}_1)\times \ldots \times \Vec^G(\mathcal{O}_N)$ - this is a direct consequence of the fact that $\delta'$ has rank $Nm$ together with \eqref{e:Aatt}. 

Besides, we take the perturbations $\alpha'$ and $Z$ sufficiently small so that the condition that $a_1^{(j')}X_1+\ldots+a_k^{(j')}X_k$ is transverse to the $G$-orbits in $V_{j'}$ for any $j'\in[j_0]$ is preserved.

By definition, if we set $X_1'=X_{1,j_0}$ and $X_2=X_{2,j_0}$, then their $N$-folds $X_1'^{(N)}$ and $X_2'^{(N)}$ form a bracket-generating system in equivariant leaves in all $V_j$ ($j\in[j_0]$), hence in $M^{(N)}_s$.  This concludes the proof of Theorem \ref{t:Nfoldbracketgenerating}.


\begin{thebibliography}{1}

\bibitem[AS04]{agrachevsachkov}
A. Agrachev and Y. Sachkov. \textit{Control theory from the geometric viewpoint}, Vol. 2. Springer Science \& Business Media, 2004.

\bibitem[AS20]{agrachevsarychev}
A. Agrachev and A. Sarychev, Control in the spaces of ensembles of points, \textit{SIAM J. Control Optim.}, 58.3, 1579-1596, 2020.

\bibitem[ADTK23]{sumformer}
S. Alberti, N. Dern, L. Thesing and G. Kutyniok, Universal Approximation for Efficient Transformers, \textit{In: Topological, Algebraic and Geometric Learning Workshops, PMLR}, 72-86, 2023. 

\bibitem[AD18]{alessandro}
F. Albertini, D. D’Alessandro, Controllability of symmetric spin networks, \textit{J. Math. Phys.}, 59.5, 2018.

\bibitem[BMR08]{bonnabel1}
S. Bonnabel, P. Martin and P. Rouchon, Symmetry-preserving observers, \textit{IEEE Trans. Autom. Control.}, 53.11, 2514-2526, 2008.

\bibitem[BMR09]{bonnabel2}
S. Bonnabel, P. Martin and P. Rouchon, Non-linear symmetry-preserving observers on Lie groups, \textit{IEEE Trans. Autom. Control.}, 54.7, 1709-1713, 2009.

\bibitem[BBLSV17]{bronstein}
M. M. Bronstein, J. Bruna, Y. LeCun, A. Szlam and P. Vandergheynst, Geometric deep learning: going beyond euclidean data, \textit{IEEE Signal Processing Magazine}, 34.4, 18-42, 2017.

\bibitem[CZDP17]{GHZ}
J. Chen, H. Zhou, C. Duan and X. Peng, Preparing Greenberger-Horne-Zeilinger and W states on a long-range Ising spin model by global controls, \textit{Physical Review A}, 95.3, 032340, 2017.

\bibitem[CLLS23]{cheng}
J. Cheng, Q. Li, T. Lin and Z. Shen, Interpolation, approximation and controllability of deep neural networks, \textit{arXiv preprint arXiv:2309.06015}, 2023.

\bibitem[CLT20]{cuchiero}
C. Cuchiero, M. Larsson and J. Teichmann, Deep neural networks, generic universal interpolation, and controlled ODEs, \textit{SIAM J. Math. Data Sci.}, 2.3, 901-919, 2020.

\bibitem[Dal23]{alessandro23}
D. D'Alessandro, Subspace controllability and Clebsch-Gordan decomposition of symmetric quantum networks, to appear in \text{SIAM J. Control Optim.}

\bibitem[Fie07]{fieldbook}
M. Field, \textit{Dynamics and symmetry}, Vol. 3. World Scientific, 2007.

\bibitem[Fri67]{frisch}
J. Frisch, Points de platitude d'un morphisme d'espaces analytiques complexes, \textit{Invent. Math.}, 4, 118-138,  1967.

\bibitem[G+23]{gerken}
J. E. Gerken, J. Aronsson, O. Carlsson, H. Linander, F. Ohlsson, C. Petersson and D. Persson, Geometric deep learning and equivariant neural networks, \textit{Artificial Intelligence Review}, 56.12, 14605-14662, 2023.

\bibitem[GLPR23a]{transformers1}
B. Geshkovski, C. Letrouit, Y. Polyanskiy and P. Rigollet, The emergence of clusters in self-attention dynamics, \textit{Advances in Neural Information Processing Systems}, 36, 2023.

\bibitem[GLPR23b]{transformers2}
B. Geshkovski, C. Letrouit, Y. Polyanskiy and P. Rigollet, A mathematical perspective on Transformers, to appear in \textit{Bull. Amer. Math. Soc}.

\bibitem[KP19]{keriven}
N. Keriven and G. Peyré, Universal invariant and equivariant graph neural networks, \textit{Advances in Neural Information Processing Systems}, 32, 2019.

\bibitem[Llo95]{logicgate}
S. Lloyd, Almost any quantum logic gate is universal, \textit{Phys. Rev. Lett.}, 75(2), 346, 1995.

\bibitem[Lob72]{lobry}
C. Lobry, Une propriété générique des couples de champs de vecteurs, \textit{Czechoslov. Math. J.}, 22.2, 230-237,1972.

\bibitem[LPM15]{luong}
M. T. Luong, H. Pham \& C. D. Manning, Effective approaches to attention-based neural machine translation, \textit{In Proceedings of the 2015 Conference on Empirical Methods in Natural Language Processing}, 1412-1421, 2015.

\bibitem[MGH22]{mahony}
R. Mahony, P. van Goor, and T. Hamel, Observer design for nonlinear systems with equivariance, \textit{Annual Review of Control, Robotics, and Autonomous Systems}, 5, 221-252, 2022.

\bibitem[Mar70]{martinet}
J. Martinet, Sur les singularités des formes différentielles, \textit{Ann. de l'Institut Fourier}, 95-178, 1970.

\bibitem[MS10]{mason}
P. Mason and M. Sigalotti, Generic controllability properties for the bilinear Schrödinger equation, \textit{Commun. Partial Diff. Equ.}, 35, 685-706, 2010.

\bibitem[SABP22]{sinkformer}
M. E. Sander, P. Ablin, M. Blondel and G. Peyré, Sinkformers: Transformers with doubly stochastic attention, \textit{International Conference on Artificial Intelligence and Statistics}, PMLR. 3515-3530 2022.

\bibitem[Sca23]{scagliotti}
A. Scagliotti, Optimal control of ensembles of dynamical systems, \textit{ESAIM: Control, Optimisation and Calculus of Variations}, 29(22), 2023.

\bibitem[TG22]{tabuada}
P. Tabuada, and B. Gharesifard, Universal approximation power of deep residual neural networks through the lens of control, \textit{IEEE Trans. Autom. Control.}, 2022.

\bibitem[Tym06]{music}
D. Tymoczko, The geometry of musical chords, \textit{Science} 313.5783, 72-74, 2006.

\bibitem[V+17]{vaswani}
A. Vaswani, N. Shazeer, N. Parmar, J. Uszkoreit, L. Jones, A. N. Gomez, L. Kaiser and I. Polosukhin, Attention is all you need, \textit{Advances in Neural Information Processing Systems}, 30, 2017.

\bibitem[VW29]{vonneumann}
J. von Neumann and E. P. Wigner, Über das Verhalten von Eigenwerten bei adiabatischen Prozessen, \textit{Physikalische Zeitschrift} 30, 467-470, 1929.

\bibitem[VBT20]{vuckovic}
J. Vuckovic, A. Baratin and R. Tachet des Combes, A mathematical theory of attention, \textit{arXiv preprint arXiv:2007.02876} (2020).

\bibitem[YBRRK19]{yun}
C. Yun, S. Bhojanapalli, A.S. Rawat, S. Reddi and S. Kumar. Are Transformers universal approximators of sequence-to-sequence functions? \textit{International Conference on Learning Representations}, 2019.
\end{thebibliography}
\end{document}